\documentclass[]{interact}

\usepackage{epstopdf}
\usepackage[caption=false]{subfig}

\usepackage[numbers,sort&compress]{natbib}
\bibpunct[, ]{[}{]}{,}{n}{,}{,}
\renewcommand\bibfont{\fontsize{10}{12}\selectfont}
\makeatletter
\def\NAT@def@citea{\def\@citea{\NAT@separator}}
\makeatother

\usepackage{amssymb}
\usepackage{amsmath}
\usepackage{amsfonts}
\usepackage{amssymb}
\usepackage{graphicx}
\usepackage{amscd}
\usepackage{xcolor}
\usepackage{float}
\usepackage{mathtools}
\usepackage{hyperref}
\newtheorem{Theorem}{Theorem}[section]
\newtheorem{Proposition}{Proposition}[section]

\newtheorem{Lemma} {Lemma}[section]

\newtheorem{Remark}{Remark}[section]

\newtheorem{Definition} {{Definition}}[section]

\def\R{\mathbb{R} }

\def\to{\rightarrow}

\def\phi{{\varphi} }
\def\WW{W_{0,\omega,\omega_{1}}^{1}L^{\Phi,\Psi}\left(\Omega\right)}
\def\W{W_{\omega,\omega_{1}}^{1}L^{\Phi,\Psi}\left(\Omega\right)}

\begin{document}

\articletype{ARTICLE}

\title{ Multiplicity of Solutions for a problem in double weighted Orlicz-Sobolev and its spectrum}

\author{
	\name{Abderrahmane Lakhdari\textsuperscript{a}\thanks{CONTACT Abderrahmane Lakhdari Email: lakhdari.abderrahmane@fst.utm.tn
		} and Nedra Belhaj Rhouma\textsuperscript{a}}
	\affil{\textsuperscript{a}Laboratory of Mathematical Analysis and Applications (LMAA-LR11-ES11), Faculty of
		Mathematical, Physical and Natural Sciences of Tunis, University of Tunis El-Manar, 2092 Tunis,
		Tunisia}
}

\maketitle

\begin{abstract}
The purpose of this paper is to investigate the existence of three different weak solutions to a nonlinear elliptic problem that is governed by the weighted $\varphi-$Laplacian operator and subjected to Dirichlet boundary conditions. We also examine the presence of sequences of variational eigenvalues in two distinct scenarios, one with and one without assuming the $\Delta_2-$condition. Our main results are obtained through technical proofs that combine a Lagrange multipliers type approach with the Ljusternik–Schnirelmann argument.
\end{abstract}

\begin{keywords}
Double weighted Orlicz-sobolev spaces; eigenvalues problem; $\Delta_2-$condition; Ljusternik–Schnirelmann theory
\end{keywords}

\section{Introduction}

Let $\Omega$ be a bounded Lipschitz domain of $\R^{N}$, $N\geq2$, we consider the following problem:

\begin{equation}\label{mainproblem0}
	\begin{cases}
		-\operatorname{div}\left(\omega(x)\frac{\varphi\left(\left|\nabla u\right| \right)}{\left|\nabla u \right| }\nabla u\right) =\lambda\omega_{1}(x)\psi(\left|u\right|)\frac{u}{\left|u\right|} & \text { in } \Omega,\\
		u=0 & \text { on }  \partial\Omega,
	\end{cases}
\end{equation}

where, $\varphi$ and $\psi$ are the derivatives of two Young functions $\Phi$ and $\Psi$. The weights $\omega(x)\geq1$ and $\omega_{1}(x)\geq1$ are respectively in \textbf{$L^{\tilde{\Psi}}(\Omega)$} and \textbf{$L^{\tilde{\Phi}}(\Omega)$}. A particular case of (\ref{mainproblem0}) can be obtained, when $\omega=\omega_{1}=1$ and $\psi=\varphi=g$. This considiration leads to the $g$-laplacian problem 
\begin{equation}\label{bahrouni}
	\begin{cases}
		-\Delta_{g}u =\lambda g(\left|u\right|)\frac{u}{\left|u\right|} & \text { in } \Omega,\\
		u=0 & \text { on }  \partial\Omega,
	\end{cases}
\end{equation}
where $\Delta_{g}u=\operatorname{div}\left(\frac{g\left(\left|\nabla u\right| \right)}{\left|\nabla u \right| }\nabla u\right)$ (see \cite{salort2022fractional}). The prototepical p(x)-laplacian case can be obtain by setting $\varphi(t)=t^{p(x)-1}$ and $\psi(t)=t^{q(x)-1}$. This lead to following problem 
\begin{equation}\label{p(x)double}
	\begin{cases}
		\Delta_{p(x)}u =\lambda\omega_{1}(x)\left| u \right|^{q(x)-2} u & \text { in } \Omega,\\
		u=0 & \text { on }  \partial\Omega,
	\end{cases}
\end{equation}
where $\Delta_{p(x)}=\operatorname{div}\left(\omega(x)\left|\nabla u \right|^{p(x)-2}\nabla u\right)$ (see \cite{unal2021compact}). More particulary if $p(x)=q(x)=p$ the operator apeared in (\ref{p(x)double}) becomes a $p$-Laplacian operator $\Delta_p u=\operatorname{div}\left(\omega(x)\left|\nabla u \right|^{p-2}\nabla u\right)$. 
\begin{Definition}
	We say that  $u\in W_{0,\omega,\omega_{1}}^{1}L^{\Phi,\Psi}\left(\Omega\right)$ is a weak solution of problem \eqref{mainproblem0} if $\forall\ v\in W_{0,\omega,\omega_{1}}^{1}L^{\Phi,\Psi}\left(\Omega\right)$
	\begin{equation}
		\int_{\Omega}\omega(x)\frac{\varphi\left(|\nabla u|\right)}{\left|\nabla u \right| }\nabla u\nabla v\,dx =\lambda\int_{\Omega}\omega_{1}(x)\psi(\left|u\right|)\frac{u}{\left|u\right|}v\,dx.
	\end{equation}
	The precise definition of $\WW$ is given in Section 3.
\end{Definition} 
The following functionals have an essential role in this study
$$
I(u)=\int_{\Omega}\omega(x)\Phi\left(\left|\nabla u\right|\right)dx,
$$
and
$$
J(u)=\int_{\Omega}\omega_{1}(x)\Psi\left(|u|\right)dx.
$$

Eigenvalues for non-homogeneous eigenproblems strongly depend on the energy level $\alpha>0$ given by the minimization problem
\begin{equation}{\label{minid}}
	\Lambda_{\alpha}:=\min\left\{ I(u):\ u\in M_{\alpha}\right\}\ \text{where}\ M_{\alpha}:=\left\lbrace u\in W_{0,\omega}^{1}L^{\Phi}\left(\Omega\right):\ J(u)=\alpha \right\rbrace.
\end{equation}   
Once we prove the existence of solution for (\ref{minid}), It is possible to prove existence of an eigenvalue for (\ref{mainproblem0}) by a
multiplier argument in the case when $I$ and $J$ are $C^{1}$.\\

We declare the following structural conditions
\begin{equation}\label{Phi1}\tag{$\Phi_{1}$}
	1<l \leq \frac{t \varphi(t)}{\Phi(t)} \leq m< \infty\quad \text { for } t>0.
\end{equation}

\begin{equation}\label{Phi2}\tag{$\Phi_{2}$}
	t\mapsto\Phi\left(\sqrt{t}\right)\ \text{is convex}.
\end{equation}

\begin{equation}\label{Psi1}\tag{$\Psi_{1}$}
	1<l_{1} \leq \frac{t \psi(t)}{\Psi(t)} \leq m_{1}<\infty \quad \text { for } t>0.
\end{equation}

\begin{equation}\label{Psi1}\tag{$\Psi_{2}$}
	\Psi\prec\prec\Phi \ \ \left(\text{see Definition \ref{definition22}} \right).
\end{equation}
We say that a Young function $\Phi$ verifies the $\Delta_{2}$-condition, if for some constant $K > 0$
$$
\Phi(2t) \leq K\Phi(t),\quad \forall t\geq 0.
$$
It is possible to show that the functions below  satisfy the $\Delta_2$-condition:
\begin{itemize}
	\item [(i)] $ \Phi(t) = (1+t^{2})^{\alpha}-1, \alpha \in (1, \frac{N}{N-2})$,
	\item [(ii)] $ \Phi(t) = t^{p}\ln(1+t), 1< \frac{-1+\sqrt{1+4N}}{2}<p<N-1, N\geq 3$,
	\item [(iii)] $\Phi(t) = \frac{1}{p}|t|^{p}$ for $p>1$,
	\item [(iv)] $\Phi(t) = \frac{1}{p}|t|^{p} + \frac{1}{q}|t|^{q}$ where $1<p<q<N$ with $q \in (p, p^{*}),$
\end{itemize}
while $\Phi(t)=(e^{t^{2}}-1)/2$ does not verify it. In \cite{Fukagai} the authors mentioned some typical examples in the fields of physics
\begin{description}
	\item[a] nonlinear elasticity: $\Phi(t)=\left(1+t^2\right)^\gamma-1, \gamma>1 / 2$,
	\item[b] plasticity: $\Phi(t)=t^\alpha(\log (1+t))^\beta, \alpha \geq 1, \beta>0$,
	\item[c] generalized Newtonian fluids: $\Phi(t)=\int_0^t s^{1-\alpha}\left(\sinh ^{-1} s\right)^\beta \mathrm{d} s, 0 \leq \alpha \leq 1$, $\beta>0$.
\end{description}  
For each $x\in\Omega$, let  $$\delta(x):=\sup\{\delta>0 \mid B(x,\delta)\subseteq\Omega\} .$$
It is clear that there exists $ x_{0}\in \Omega$ such that
$B(x_{0},D)\subseteq\Omega$
where $D=\displaystyle\sup_{x\in\Omega}\delta(x)$.
Our initial outcomes are summarized as follows
\begin{Theorem}\label{T0}
	Let $\Phi$ satisfies $(\Phi_{1})$, $(\Phi_{2})$ and let $\Psi$ satisfies $(\Psi_{1})$, $(\Psi_{2})$. Assume that there exist $r>0$ and $d>0$ such that
	\begin{equation*}\label{r4} r<\min\left\lbrace\left\|\frac{2d}{D}\right\|_{L_{\omega}^{\Phi}\left(\Omega\right)}^{l},\left\|\frac{2d}{D}\right\|_{L_{\omega}^{\Phi}\left(\Omega\right)}^{m}  \right\rbrace
	\end{equation*}
	and
	\begin{equation*}
		\begin{aligned}
			\tilde{w}_r &:=\max\left\lbrace C^{l_{1}}_{1},C^{m_{1}}_{1} \right\rbrace\max\left\lbrace \max\left\lbrace\left| r \right|^{\frac{l_{1}}{l}},\left| r\right|^{\frac{l_{1}}{m}}\right\rbrace,\max\left\lbrace\left| r \right|^{\frac{m_{1}}{l}},\left| r\right|^{\frac{m_{1}}{m}}  \right\rbrace\right\rbrace\\
			&<\gamma_d:=\frac{\min\left\lbrace\left|d\right|^{l_{1}},\left|d\right|^{m_{1}}\right\rbrace\Psi\left(1\right)\frac{\pi^{\frac{N}{2}}}{\frac{N}{2}\Gamma(\frac{N}{2})}\left(\frac{D}{2}\right)^{N} }{\left(2N\right)^{m}\max\left\lbrace\left\|\frac{d}{D}\right\|_{L_{\omega}^{\Phi}\left(\Omega\right)}^{l},\left\|\frac{d}{D}\right\|_{L_{\omega}^{\Phi}\left(\Omega\right)}^{m} \right\rbrace}.
		\end{aligned}
	\end{equation*}
	Then, there exist $v_{d}\in\WW$, such that for every $\lambda\in\Lambda_r:=\left( \frac{I(v_{d})}{J(v_{d})}, \frac{r}{\displaystyle\sup_{I(u)\leq r}J(u)}\right)$ the problem \eqref{mainproblem0} admits at least three weak solutions.
\end{Theorem}
\smallskip
For the next result, for any $k\in\mathbb{N}$, we set 
$$
\mathcal{C}_k:=\{ K\subset M_\alpha \text{ compact, symmetric with } I(u)>0 \text{ on } K  \text{ and } \gamma(K)\geq k\},
$$
where $\gamma(K)$ denotes the Krasnoselskii genus of $K$ defined by 
$$
\gamma(K):=\inf\{p\in\mathbb{N}\colon \exists h\colon K\to \R^p\setminus\{0\}\ \text{such that}\ h \text{ is continuous and odd} \}.
$$
Now, we are in position to give our second result in the following Theorem
\begin{Theorem}\label{T05}
	Let $\Phi$ and $\Psi$ two Young functions satisfy $(\Phi_{1})$ and $(\Psi_{1})$, $(\Psi_{2})$ respectively. For any $\alpha>0$
	there is a non-negative sequence $\{\lambda_{k,\alpha}\}_{k\in\mathbb{N}}$ of eigenvalues of (\ref{mainproblem0}) with  eigenfunctions $\{u_{k,\alpha}\}_{k\in \mathbb{N}}\subset W_{0,\omega,\omega_{1}}^{1}L^{\Phi,\Psi}\left(\Omega\right)$ satisfying 
	\begin{equation}\label{E}
		J(u_{k,\alpha})=\alpha,\ \text{and} \ \ I(u_{k,\alpha}):=c_{k,\alpha} = \sup_{K\in \mathcal{C}_k}\inf_{u\in K}I(u).
	\end{equation}
\end{Theorem}
In this sequel, we give some results without assuming the $\Delta_{2}$-condition.
\begin{Theorem}\label{T1}
	Let $\Psi$ satisfies $(\Psi_{2})$. For each $\alpha>0$, there exists $u_{\alpha}\in \WW$ solution of the minimization Problem (\ref{minid}).
\end{Theorem}
\begin{Theorem}\label{T2}
	Let $\Psi$ satisfies $(\Psi_{2})$. If $u_{\alpha}$ solves the minimization Problem (\ref{minid}) then there exists $\lambda_{\alpha}>0$ such that $u_{\alpha}$ is a weak solution of the Problem (\ref{mainproblem0}).
\end{Theorem}
Next, we define the spectrum $\Sigma_{\left(\ref{mainproblem0}\right)}$ by
\begin{equation*}
	\Sigma_{\left(\ref{mainproblem0}\right)}:=\left\lbrace\lambda\in\R; \ \exists u\in W_{\omega,\omega_{1}}^{1}L^{\Phi,\Psi}\left(\Omega\right)\ \text{solution of Problem (\ref{mainproblem0})} \right\rbrace
\end{equation*}
\begin{Theorem}\label{T3}
	Let $\Psi$ satisfies $(\Psi_{2})$. The spectrum $\Sigma_{\left(\ref{mainproblem0}\right)}$ is a closed subset of $\R$.
\end{Theorem}

\section{Preliminaries}\label{sec2}

In this section we recall some properties of Orlicz and Orlicz-Sobolev spaces, which can be found in \cite{Krasnoselskii, KJF, RR}. If the reader is familiar with the topic, he or she can skip the section and go directly to the next sections.

\subsection{Young function}

\begin{Definition}
	A function $\Phi: \mathbb{R}_+\rightarrow\mathbb{R_{+}}$ is termed a Young function if it admits the integral representation $\Phi(t)=\int_{0}^{t}\varphi(s)ds,$ where the right continuous function $\varphi$ defined on $[0, \infty)$ is assumed to satisfy the following conditions:
	\begin{itemize}
		\item [($\varphi1$)] $\varphi(0) = 0,\ \varphi(s) > 0\ \text{for}\ s > 0,$
		\item [($\varphi2$)] $\varphi$ is non-decreasing,
		\item [($\varphi3$)] $\underset{s\to\infty}{\lim}\varphi(s)=\infty.$
	\end{itemize}
\end{Definition}

This condition states that a Young function $\Phi$ is convex, nonnegative, and strictly growing on the interval $[0,\infty]$. The convexity of $\Phi$ can follows the following estimations:
\begin{equation}\label{es1}
	\Phi(\alpha t)\leq \alpha \Phi(t), \ \text{if} \ \alpha\in\left[0,1\right], t\geq0,
\end{equation}
and
\begin{equation}\label{es2}
	\Phi(\beta t)\geq \beta \Phi(t), \ \text{if} \ \beta\in\left(1,\infty\right), t\geq0.
\end{equation}
In \cite[Theorem 4.1]{Krasnoselskii}, it is shown that the $\Delta_2-$ condition is equivalent to
$$
\frac{t \varphi(t)}{\Phi(t)} \leq m \quad \text { for } t>0,
$$
for some $m>1$.
The complementary function $\widetilde{\Phi}$ associated with $\Phi$ is given by its Legendre's transformation, that is,
$$
\widetilde{\Phi}(s) = \max_{t\geq 0}\{ st - \Phi(t)\}, \quad  \mbox{for} \quad s\geq0.
$$
The functions $\Phi$ and $\widetilde{\Phi}$ are complementary each other. Moreover, we also have a Young type inequality given by
$$
st \leq \Phi(t) + \widetilde{\Phi}(s), \quad \forall s, t\geq0.
$$
Let $\varphi^{-1}$ be the inverse of $\varphi$, then the complementary Young function admits the following formula
\begin{equation}\label{2.2}
	\widetilde{\Phi}(\tau)=\int_{0}^{\tau} \varphi^{-1}(t) dt.
\end{equation}
The accompanying relations are obviously helpful
\begin{equation}\label{messt}
	\Phi(2t)=\int_{0}^{2t}\varphi(\tau)d\tau>\int_{t}^{2t}\varphi(\tau)d\tau>t\varphi(t),
\end{equation}

\begin{equation}\label{messt1}
	\Phi(t)=\int_{0}^{t}\varphi(\tau)d\tau\leq t\varphi(t)
\end{equation}
and
\begin{equation}\label{messt2}
	\widetilde{\Phi}(\varphi(t)) \leq t \varphi(t) \leq \Phi(2 t)
\end{equation}

\begin{Lemma}\label{important2}
	Let $\Phi$ be a Young function satisfying $\left(\Phi_1\right)$ and $\left(\Phi_2\right)$
	Then for every $a, b \in \mathbb{R}$,
	$$
	\frac{\Phi(|a|)+\Phi(|b|)}{2} \geq \Phi\left(\left|\frac{a+b}{2}\right|\right)+\Phi\left(\left|\frac{a-b}{2}\right|\right) .
	$$
\end{Lemma}
\begin{Remark}
	In the case when $\Phi(t)=t^{p}$, the condition $(\Phi_{2})$ cover just the case when $p\geq2$.
\end{Remark}
\begin{Lemma}\label{important3}
	Let $\Phi$ be an Young function satisfying $\left(\Phi_1\right)$ such that $\varphi=\Phi^{\prime}$ and denote by $\tilde{\Phi}$ its complementary function. Then
	$$
	\widetilde{\Phi}(\varphi(t)) \leq m \Phi(t)
	$$
	holds for any $t \geq 0$.
\end{Lemma}

\begin{Lemma} \label{lema.prop}
	Let $\Phi$ be a Young function satisfying $(\Phi_{1})$ and $a,b\geq 0$. Then
	\begin{align*}
		&\min\{ a^{l}, a^{m}\} \Phi(b) \leq \Phi(ab)\leq   \max\{a^{l},a^{m}\} \Phi(b),\tag{$L_1$}\label{L1}\\
		&\Phi(a+b)\leq \mathbf{C} (G(a)+G(b)) \quad \text{with } \mathbf{C}=  \mathbf{C}(m),\tag{$L_2$}\label{L2}\\
		&\tilde {\Phi}(a+b)\leq \mathbf{\tilde C} (\tilde {\Phi}(a)+\tilde {\Phi}(b)) \quad \text{with }  \mathbf{\tilde C}=  \mathbf{\tilde C}(l),\tag{$L_3$}\label{L2'}\\
		&\Phi \text{ is Lipschitz continuous}. \tag{$L_4$}\label{L_3}
	\end{align*}
\end{Lemma}

We need the notion of comparison between Young functions.

\begin{Definition}\label{definition22}
	Given two Young functions $\Phi$ and $\Psi$, we say that $\Phi \leq \Psi$ if there exists a constant $c>0$ and $t_0>0$ such that $A(t) \leq B(c t)$, for every $t \geq t_0$.
	
	Whenever $\Phi \leq \Psi$ and $\Psi \leq \Phi$ we say that $\Phi$ and $\Psi$ are equivalent Young functions and this fact will be denoted by $\Phi \sim \Psi$.
	
	Finally, we say that $\Phi$ decreases essentially more rapidly than $\Psi$, denoted $\Phi\prec\prec\Psi$, if for any $c>0$,
	$$
	\lim _{t \rightarrow \infty} \frac{\Phi(ct)}{\Psi(t)}=0 .
	$$
\end{Definition}

\subsection{Orlicz and Orlicz-Sobolev spaces.}
The Orlicz classe $K^{\Phi}(\Omega)$ is defined by
$$
K^{\Phi}(\Omega):=\left\{ \text{u real value measurable function on} \Omega:\ \underset{\Omega}{\int}\Phi(\left| u(x)\right| )dx<\infty \right\}.
$$
It is notable that this Orlicz class is a vector space if and only if $\Phi$ satisfies the $\Delta_2$-condition.	
The orlicz space $L^{\Phi}(\Omega)$ is the linear span of $K^{\Phi}(\Omega)$. 
The definition of $L^{\Phi}(\Omega)$ follows that
$$
K^{\Phi}(\Omega)\subset L^{\Phi}(\Omega).
$$
Furthermore, we have the correspondence if and only if $\Phi$ satisfies the $\Delta_2$condition.

The space $L^{\Phi}(\Omega)$  is a Banach space endowed with the Luxemburg norm
\begin{equation}\label{norm1}
	\parallel u\parallel_{L^{\Phi}(\Omega)}:=\inf\left\{ \xi>0,\;\; \int_{\Omega}\Phi\left(\frac{|u(x)|}{\xi}\right)dx \leq1\right\}.
\end{equation}
$L^{\Phi}(\Omega)$  is separable if and only if $\Phi$ satisfy the $\Delta_2$condition and it is reflexive if and only if $\Phi$ and $\tilde{\Phi}$ satisfy the $\Delta_2$condition.

We consider $E^{\Phi}(\Omega)$ as the closure in $L^{\Phi}(\Omega)$ of all bounded measurable functions. Thus, $E^{\Phi}(\Omega)$ is a separable Banach space and
$$
E^{\Phi}(\Omega)\subset K^{\Phi}(\Omega),
$$
with equality, if and only if $\Phi$ satisfies the $\Delta_2$-condition. We have the following H\"{o}lder's type inequality
$$
\underset{\Omega}{\int}\left|uv\right|dx\leq2\parallel u\parallel_{L^{\Phi}(\Omega)}\parallel u\parallel_{L^{\tilde{\Phi}}(\Omega)}.
$$
The dual spaces of $E^{\Phi}(\Omega)$ and $E^{\tilde{\Phi}}(\Omega)$ are given by
$$
\left(E^{\widetilde{\Phi}}\right)^{\prime}=L^\Phi, \quad \text { and } \quad\left(E^\Phi\right)^{\prime}=L^{\widetilde{\Phi}}.
$$
The Orlicz-Sobolev space is defined by
$$
W^{1}L^{\Phi}\left(\Omega\right):=\left\{u \in L^{\Phi}\left(\Omega\right);\;\; \nabla u\in L^{\Phi}\left(\Omega\right) \right\}.
$$
$W^{1}L^{\Phi}\left(\Omega\right)$ is a Banach space equipped with the following norm
$$
\parallel u\parallel_{W^{1}L^{\Phi}}:=\parallel u\parallel_{L^{\Phi}\left(\Omega\right)}+\parallel \nabla u\parallel_{L^{\Phi}\left(\Omega\right)},
$$
$W^{1}L^{\Phi}\left(\Omega\right)$ is separable if and only if $\Phi$ satisfy the $\Delta_2$-condition and $W^{1}L^{\Phi}\left(\Omega\right)$ is reflexive if and only if $\Phi$ and $\tilde{\Phi}$ satisfy the $\Delta_2$-condition. The space $W^{1}E^{\Phi}\left(\Omega\right)$ is defined in an analogous way and is a closed subspace of $W^{1}L^{\Phi}(\Omega)$. Furthermore, $W^{1}E^{\Phi}\left(\Omega\right)$ is a separable Banach space.

We define $W_0^1 L^\Phi(\Omega)$ as the weak* closure of $C_c^{\infty}(\Omega)$ in $W^1 L^\Phi(\Omega)$, hence $W_0^1 L^\Phi(\Omega)$ is a weak* closed subset of the dual of a separable space.

In order for the Sobolev immersion theorem to hold, one need to impose some grawth conditions on $\Phi$. Following \cite{cianchi2004optimal}, we require $\Phi$ to verify
\begin{equation}\label{emdh1}
	\int_K^{\infty}\left(\frac{t}{\Phi(t)}\right)^{\frac{1}{N-1}} d t=\infty,
\end{equation}
\begin{equation}\label{embh2}
	\int_0^\gamma\left(\frac{t}{\Phi(t)}\right)^{\frac{1}{N-1}} d t<\infty,
\end{equation}
for some constants $K, \gamma>0$.
Given a Young function $\Phi$ that satisfes \eqref{emdh1} and \eqref{embh2} its Orlicz-Sobolev conjugate is defined as
\begin{equation}\label{HN}
	\Phi_N(t)=\Phi \circ H^{-1}(t),
\end{equation}
where
$$
H(t)=\left(\int_0^t\left(\frac{\tau}{\Phi(\tau)}\right)^{\frac{1}{N-1}} d \tau\right)^{\frac{N-1}{N}} .
$$
Notice that the function given in the examples satisfy (\ref{emdh1}) and (\ref{embh2}).\\ 
The following fundamental Orlicz-Sobolev embedding Theorem can be found in \cite{cianchi2004optimal}.

\begin{Theorem}\label{Cianchiembedding}
	Let $\Phi$ be a Young function satisfying \eqref{emdh1} and \eqref{embh2} and let $\Phi_N$ be defined in \eqref{HN}. Then the embedding $W_0^{1}L^{\Phi}(\Omega) \hookrightarrow L^{\Phi_N}(\Omega)$ is continuous. Moreover, the Young function $\Phi_N$ is optimal in the class of Orlicz spaces.
	
	Finally, given $\Psi$ any Young funcion, the embedding $W_0^{1}L^{\Phi}(\Omega) \hookrightarrow L^\Psi(\Omega)$ is compact if and only if $\Psi <<\Phi_N$.
\end{Theorem}

It is easy to see that $\Phi<< \Phi_N$ and hence $W_0^{1}L^{\Phi}(\Omega) \subset L^\Phi(\Omega)$ is compact.

\subsection{Weighted Orlicz spaces}\cite{bloom1994weighted,el2019nonlinear,krbec1992imbedding,krbec1991imbeddings,osanccliol2014inclusions}\\
The weighted Orlicz space $L_{\omega}^{\Phi}(\Omega)$ is the linear span of the orlicz classe $K_{\omega}^{\Phi}(\Omega)$ which given by
\begin{equation*}
	K_{\omega}^{\Phi}(\Omega):=\left\lbrace \text{u measurable and defined in } \Omega:\ \underset{\Omega}{\int}\omega(x)\Phi(\left| u(x)\right| )dx<\infty \right\rbrace,
\end{equation*}
it is notable that this orlicz classe is a vector space if and only if $\Phi$ satisfies \textbf{$\Delta_{2}$-condition}. The definition of $L^{M}(\Omega)$ follows that
\begin{equation*}
	K_{\omega}^{\Phi}(\Omega)\subset L_{\omega}^{\Phi}(\Omega).
\end{equation*}
Furthermore, we have the correspondence if and only if $\Phi$ satisfies \textbf{$\Delta_{2}$-condition}.\\

The space $L_{\omega}^{\Phi}(\Omega)$  is a Banach space endowed with the Luxemburg norm
\begin{equation}\label{norm1}
	\parallel u\parallel_{L_{\omega}^{\Phi}(\Omega)}:=\inf\left\{ \xi>0,\;\; \int_{\Omega}\omega(x)\Phi\left(\frac{|u(x)|}{\xi}\right)dx \leq1\right\},
\end{equation}
separable if and only if $\Phi$ satisfy the \textbf{$\Delta_{2}$-condition} and it is reflexive if and only if $\Phi$ and $\tilde{\Phi}$ satisfy the \textbf{$\Delta_{2}$-condition}.
\begin{Definition}
	We consider $E_{\omega}^{\Phi}(\Omega)$ as The closure in $L_{\omega}^{\Phi}(\Omega)$ of all bounded measurable functions.
\end{Definition}
\begin{Remark}
	\begin{itemize}
		\item[(i)] $E_{\omega}^{\Phi}(\Omega)$ is a separable Banach space and $E_{\omega}^{\Phi}(\Omega)\subset K_{\omega}^{\Phi}(\Omega)$
		\item[(ii)] $E_{\omega}^{\Phi}(\Omega)=K_{\omega}^{\Phi}(\Omega)$, if and only if $\Phi$ satisfies \textbf{$\Delta_{2}$-condition}.
	\end{itemize}
\end{Remark} 
\begin{Remark}
	We have the following H$\ddot{o}$lder’s type inequality
	\begin{equation}\label{weithedh}
		\underset{\Omega}{\int}\omega\left|uv\right|dx\leq2\parallel u\parallel_{L_{\omega}^{\Phi}(\Omega)}\parallel u\parallel_{L_{\omega}^{\tilde{\Phi}}(\Omega)}.
	\end{equation}
\end{Remark}
The dual spaces of $E_{\omega}^{\Phi}(\Omega)$ and $E_{\omega}^{\widetilde{\Phi}}(\Omega)$ are given by
\begin{equation}
	\left(E_{\omega}^{\widetilde{\Phi}}\right)^{\prime}=L_{\omega}^\Phi, \quad \text { and } \quad\left(E_{\omega}^\Phi\right)^{\prime}=L_{\omega}^{\widetilde{\Phi}}.
\end{equation}
The weighted Orlicz-Sobolev space is defined by
\begin{equation*}
	W_{\omega}^{1}L^{\Phi}\left(\Omega\right):=\left\{u \in L^{\Phi}\left(\Omega\right);\;\; \left\|\nabla u\right\|_{L_{\omega}^{\Phi}\left(\Omega\right)}<\infty \right\}.
\end{equation*}
It is a Banach space equipped with the following norm
\begin{equation*}
	\parallel u\parallel_{W_{\omega}^{1}L^{\Phi}}:=\parallel u\parallel_{L^{\Phi}\left(\Omega\right)}+\parallel \nabla u\parallel_{L_{\omega}^{\Phi}\left(\Omega\right)},
\end{equation*}
separable if and only if $\Phi$ satisfy the \textbf{$\Delta_{2}$-condition} and it is reflexive if and only if $\Phi$ and $\tilde{\Phi}$ satisfy the \textbf{$\Delta_{2}$-condition}. The space $W_{\omega}^{1}E^{\Phi}\left(\Omega\right)$ is defined in an analogous way and is a closed subspace of $W_{\omega}^{1}L^{\Phi}(\Omega)$. Furthermore, $W_{\omega}^{1}E^{\Phi}\left(\Omega\right)$ is a separable Banach space.\\
\begin{Proposition}\label{important}
	Let $\Phi$ be a Young function satisfying $\left(\Phi_1\right)$. Then for all $u\in L_{\omega}^{\Phi}\left(\Omega\right)$
	$$\min \left\{\parallel u\parallel_{L_{\omega}^{\Phi}\left(\Omega\right)}^{l}, \parallel u\parallel_{L_{\omega}^{\Phi}\left(\Omega\right)}^{m}\right\} \leq \int_{\Omega}\omega(x)\Phi(\left| u\right| )dx$$
	and
	
	$$\int_{\Omega}\omega(x)\Phi(\left|u\right|)dx\leq \max \left\{\parallel u\parallel_{L_{\omega}^{\Phi}\left(\Omega\right)}^{l}, \parallel u\parallel_{L_{\omega}^{\Phi}\left(\Omega\right)}^{m}\right\}$$.
\end{Proposition}	
%
\section{Double Weighted Orlicz-Sobolev spaces}

We define the double weighted Orlicz-Sobolev space as follows
$$
W_{\omega,\omega_{1}}^{1}L^{\Phi,\Psi}\left(\Omega\right):=\left\{u \in L_{\omega_{1}}^{\Psi}\left(\Omega\right),\;\; \left|\nabla u\right|  \in L_{\omega}^{\Phi}\left(\Omega\right)\right\},
$$
equipped with the following norm
\begin{equation}\label{norm2}
	\parallel u\parallel_{W_{\omega,\omega_{1}}^{1}L^{\Phi,\Psi}\left(\Omega\right)}:=\parallel u\parallel_{L_{\omega_{1}}^{\Psi}\left(\Omega\right)}+\parallel \nabla u\parallel_{L_{\omega}^{\Phi}\left(\Omega\right)}.
\end{equation}

\begin{Proposition}
	\begin{itemize}
		\item [(i)] The Space $W_{\omega,\omega_{1}}^{1}L^{\Phi,\Psi}\left(\Omega\right)$ is Banach with respect to the norm $\parallel .\parallel_{W_{\omega,\omega_{1}}^{1}L^{\Phi,\Psi}\left(\Omega\right)}$.
		\item [(ii)] The Space $W_{\omega,\omega_{1}}^{1}L^{\Phi,\Psi}\left(\Omega\right)$ is separable if $\Phi$ and $\Psi$ satisfy the $\Delta_2$-condition, and $W_{\omega,\omega_{1}}^{1}L^{\Phi,\Psi}\left(\Omega\right)$ is reflexive if in addition $\widetilde{\Phi}$ and $\widetilde{\Psi}$ satisfy the $\Delta_2$-condition.
	\end{itemize}
\end{Proposition}

\begin{proof}
	$(i)$ Let $\left\lbrace u_{n}\right\rbrace_{n}$ be a Cauchy sequence in $W_{\omega,\omega_{1}}^{1}L^{\Phi,\Psi}\left(\Omega\right)$.
	Then, $u_{n}$ and $\frac{\partial u_{n}}{\partial x_{i}}, \ i=1,...,N$ are Cauchy sequences in $ L_{\omega_{1}}^{\Psi}\left(\Omega\right)$ and $L_{\omega}^{\Phi}\left(\Omega\right)$, respectively. Since  the spaces $ L_{\omega_{1}}^{\Psi}\left(\Omega\right)$ and $L_{\omega}^{\Phi}\left(\Omega\right)$ are Banach, then
	$$
	u_{n}\longrightarrow u\ \text{in}\ L_{\omega_{1}}^{\Psi}\left(\Omega\right)
	$$
	$$
	\frac{\partial u_{n}}{\partial x_{i}}\longrightarrow v_{i} \ \text{in}\ L_{\omega}^{\Phi}\left(\Omega\right), \ i=1,...,N.
	$$
	Now, we will show that each $v_{i}$ coincides with $\frac{\partial u_{n}}{\partial x_{i}}$ in the distributional sense. Indeed, For every $\rho\in C^{\infty}_{0}(\Omega),$ $\rho>0$, we have
	$$
	\begin{aligned}
		\left|\int_{\Omega}u_n\rho dx-\int_{\Omega}u\rho dx\right|&\leq\int_{\Omega}\left|u_{n}-u\right|\rho dx\leq\int_{\Omega}\omega_{1}^{2}\left|u_{n}-u\right|\rho dx\\
		&\leq2\left\|u_{n}-u \right\|_{L_{\omega_{1}}^{\Psi}\left(\Omega\right)}\left\|\omega_{1}\rho \right\|_{L_{\omega_{1}}^{\tilde{\Psi}}\left(\Omega\right)}\\
		&\leq2\left\|u_{n}-u \right\|_{L_{\omega_{1}}^{\Psi}\left(\Omega\right)}\left\|\omega_{1} \right\|_{L_{\omega_{1}}^{\tilde{\Psi}}\left(supp\left\lbrace\rho\right\rbrace \right)}\left\|\rho \right\|_{\infty}
	\end{aligned}
	$$
	Since $u_{n}\longrightarrow u\ \text{in}\ L_{\omega_{1}}^{\Psi}\left(\Omega\right)$, then for all $\rho\in C^{\infty}_{0}$
	$$
	\int_{\Omega}u_{n}\rho dx\longrightarrow\int_{\Omega}u\rho dx.
	$$
	Similarly, since $\frac{\partial u_{n}}{\partial x_{i}}\longrightarrow v_{i}\ \text{in}\ L_{\omega}^{\Phi}\left(\Omega\right), \ i=1,...,N.$, we get for all $\rho\in C^{\infty}_{0}$
	\begin{equation*}
		\int_{\Omega}\frac{\partial u_{n}}{\partial x_{i}}\rho dx\longrightarrow\int_{\Omega}v_{i}\rho dx.
	\end{equation*}
	This yields
	\begin{equation*}
		\begin{aligned}
			\int_{\Omega}v_{i}\rho dx&=\underset{n\longrightarrow\infty}{\lim}\int_{\Omega}\frac{\partial u_{n}}{\partial x_{i}}\rho dx\\
			&=-\underset{n\longrightarrow\infty}{\lim}\int_{\Omega}u_{n}\frac{\partial \rho}{\partial x_{i}} dx\\
			&=-\int_{\Omega}u\frac{\partial \rho}{\partial x_{i}} dx
		\end{aligned}
	\end{equation*}
	It follows that $v_{i}=\frac{\partial u}{\partial x_{i}}$. Hence, $u_{n}\longrightarrow u$ in $W_{\omega,\omega_{1}}^{1}L^{\Phi,\Psi}\left(\Omega\right)$. The proof is complete.

	$(ii)$ Consider
	$$
	\begin{aligned}
		B:W_{\omega,\omega_{1}}^{1}L^{\Phi,\Psi}\left(\Omega\right)&\longrightarrow L^{\Psi}_{\omega_{1}}(\Omega)\times \left( L^{\Phi}_{\omega}(\Omega)\right)^{N} \\
		u&\longmapsto\left(u,\nabla u\right)
	\end{aligned}
	$$
	By the definition of the Norm $\left\|u\right\|_{W_{\omega,\omega_{1}}^{1}L^{\Phi,\Psi}\left(\Omega\right)}$ the map $B$ is an isometry.
	Then the space $W_{\omega,\omega_{1}}^{1}L^{\Phi,\Psi}\left(\Omega\right)$ can be isometrically identified by $L^{\Psi}_{\omega_{1}}(\Omega)\times \left( L^{\Phi}_{\omega}(\Omega)\right)^{N}$ which is separable if $\Phi$ with $\Psi$ satisfies the $\Delta_2$-condition and
	is reflexive if in addition $\tilde{\Phi}$ and $\tilde{\Psi}$ satisfy the $\Delta_2$-condition. The proof is complete.
\end{proof}

We define the following space
$$
W_{\omega,\omega_{1}}^{1}E^{\Phi,\Psi}\left(\Omega\right):=\left\{u \in E_{\omega_{1}}^{\Psi}\left(\Omega\right),\;\; \left|\nabla u\right|  \in E_{\omega}^{\Phi}\left(\Omega\right)\right\}.
$$

$W_{\omega,\omega_{1}}^{1}E^{\Phi,\Psi}\left(\Omega\right)$ is a closed and separable subspace of $\W.$

Now, we define the space $W_{0,\omega,\omega_{1}}^{1}L^{\Phi,\Psi}\left(\Omega\right)$ as the weak* closure of $C_{c}^{\infty}\left(\Omega \right).$

\begin{Proposition}\label{embe}
	The embedding $W_{\omega,\omega_{1}}^{1}L^{\Phi,\Psi}\left(\R^{N}\right)\hookrightarrow L^{\Phi_{N}}(\mathbb{R}^{N})$ is continuous for all $\Psi\prec\prec\Phi_{N}$, where $\Phi_{N}$ is given by \eqref{HN}. In particular, the embedding
	$W_{\omega,\omega_{1}}^{1}L^{\Phi,\Psi}\left(\R^{N}\right)\hookrightarrow L_{\omega}^{\Phi}(\mathbb{R}^{N})$
	and
	$W_{\omega,\omega_{1}}^{1}L^{\Phi,\Psi}(\mathbb{R}^{N})\hookrightarrow L_{\omega_{1}}^{\Psi}(\mathbb{R}^{N})$ are continuous. Moreover, under (\ref{emdh1}) and (\ref{embh2}) the embeddings $W_{\omega,\omega_{1}}^{1}L^{\Phi,\Psi}\left(\R^{N}\right)\hookrightarrow E_{\omega}^{\Phi}(\mathbb{R}^{N})$ and $W_{\omega,\omega_{1}}^{1}L^{\Phi,\Psi}\left(\R^{N}\right)\hookrightarrow E_{\omega_{1}}^{\Psi}(\mathbb{R}^{N})$ are compact. 
\end{Proposition}
\begin{proof}
	The proof is similar to the proof of Theorem \ref{Cianchiembedding}.
\end{proof}
%
%
\begin{Proposition}[Poincar\'{e} inequality]\label{poincare}
	Let $u\in W_{0,\omega,\omega_{1}}^{1}L^{\Phi,\Psi}\left(\R^{N}\right)$. If $\Psi\prec\prec\Phi$, then there exist a constant $C>0$ such that
	$$
	\left\|u\right\|_{L^{\Psi}_{\omega_{1}}(\Omega)}\leq C\left\|\nabla u \right\|_{L^{\Phi}_{\omega}(\Omega)}.
	$$
\end{Proposition}

\begin{proof}
	Since $\Psi\prec\prec\Phi$, then there exist $C>0$ such that
	$$
	\left\|u\right\|_{L_{\omega_{1}}^{\Psi}(\Omega)}\leq C\left\| u\right\|_{L_{\omega}^{\Phi}(\Omega)}.
	$$
	From \cite[Lemma 5.7]{gossez1974nonlinear}, there exist $C_{1}>0$ such that
	$$
	\left\|u\right\|_{L_{\omega_{1}}^{\Psi}(\Omega)}\leq C\left\| u\right\|_{L_{\omega}^{\Phi}(\Omega)}\leq C_{1}\left\| \nabla u\right\|_{L_{\omega}^{\Phi}(\Omega)}.
	$$
	The proof is complete.
\end{proof}
As a usual application of the Proposition \ref{poincare}, the norm $\left\|\nabla . \right\|_{L^{\Phi}_{\omega}(\Omega)}$ is equivalant to the norm $\left\| .\right\|_{W_{\omega,\omega_{1}}^{1}L^{\Phi,\Psi}\left(\Omega\right)}$ in $W_{0,\omega,\omega_{1}}^{1}L^{\Phi,\Psi}\left(\Omega\right)$.\\

Now, we give the following technical Lemmas.

\begin{Lemma}\label{algebric1}
	Given $v\in W_{0,\omega,\omega_{1}}^{1}L^{\Phi,\Psi}\left(\Omega\right)$, then $\delta v\in K^{\Psi}_{\omega_{1}}(\Omega)$, for all $\delta>0$. Moreover there exists $T>0$ such that 
	\begin{equation}\label{*}
		\int_{\Omega}\omega_{1}(x)\Psi(\delta v(x))dx
		\leq2\Psi(T)\left\|\omega_{1}\right\|_{L_{\omega_{1}}^{\tilde{\Psi}}\left(\Omega_{k}\right)}\left\|1\right\|_{L_{\omega_{1}}^{\Psi}\left(\Omega_{k}\right)}
		+1.	
	\end{equation}
\end{Lemma}

\begin{proof}
	Since $v\in W_{0,\omega,\omega_{1}}^{1}L^{\Phi,\Psi}\left(\Omega\right)$, then $v\in L_{\omega_{1}}^{\Psi}\left(\Omega\right)$ and by Proposition (embedding), we get $v\in L_{\omega_{1}}^{\Psi_{*}}\left(\Omega\right)$. In addition, $\Psi$ grows essentially more slowly than
	$\Psi_{*}$ as we proved in Proposition (grow). Therefore, we can choose positive constants $K$ and $T$ such that $\left\|u \right\|_{L_{\omega_{1}}^{\Psi_{*}}\left(\Omega\right)}\leq K$ and $\Psi(t)\leq\Psi_{*}\left(\frac{t}{\delta K} \right)$ for $t>T$. Now, we consider the following subset
	$$
	\Omega_{k}:=\left\lbrace x;\ v(x)\leq \frac{T}{\delta} \right\rbrace.
	$$
	Then,
	$$
	\begin{aligned}
		\int_{\Omega}\omega_{1}(x)\Psi(\delta v(x))dx&\leq\int_{\Omega_{k}}\omega_{1}(x)\Psi(\delta v(x))dx+\int_{\Omega/\Omega_{k}}\omega_{1}(x)\Psi_*\left(\frac{v(x)}{K}\right) dx\\
		&\leq2\Psi(T)\left\|\omega_{1}\right\|_{L_{\omega_{1}}^{\tilde{\Psi}}\left(\Omega_{k}\right)}\left\|1\right\|_{L_{\omega_{1}}^{\Psi}\left(\Omega_{k}\right)}
		+\int_{\Omega}\omega_{1}(x)\Psi_*\left(\frac{v(x)}{K}\right)dx\\
		&\leq2\Psi(T)\left\|\omega_{1}\right\|_{L_{\omega_{1}}^{\tilde{\Psi}}\left(\Omega_{k}\right)}\left\|1\right\|_{L_{\omega_{1}}^{\Psi}\left(\Omega_{k}\right)}
		+1.
	\end{aligned}
	$$
	Hence, $\delta v\in K^{\Psi}_{\omega_{1}}(\Omega)$.
	
\end{proof}
\begin{Remark}
	The proof for $u\in W_{0,\omega,\omega_{1}}^{1}E^{\Phi,\Psi}\left(\Omega\right)$ run analogously.
\end{Remark}	
\begin{Lemma}
	Given $v\in W_{0,\omega,\omega_{1}}^{1}L^{\Phi,\Psi}\left(\Omega\right)$, then $\psi(\left|v\right|)\in K^{\widetilde{\Psi}}_{\omega_{1}}(\Omega)$.
\end{Lemma}

\begin{proof}
	Using the fact that $\psi$ is increasing, (\ref{2.2}) and (\ref{messt}) we have
	\begin{equation}\label{algebric2}
		\begin{aligned}
			\int_{\Omega}\omega_{1}(x)\widetilde{\Psi}\left(\psi(\left|v\right|)\right)dx&=\int_{\Omega}\omega_{1}\left(\int_{0}^{\psi\left( \left|v\right|\right)}\psi^{-1}(s)ds\right)dx\\
			&\leq\int_{\Omega}\omega_{1}(x)\left|v\right|\psi\left( \left|v\right|\right)dx\\
			&\leq\int_{\Omega}\omega_{1}(x)\Psi\left(2\left|v\right|\right)dx.
		\end{aligned}
	\end{equation}
	By Lemma \ref{algebric1}, we know that $2\left|v\right|\in K^{\Psi}_{\omega_{1}}(\Omega)$, this together with the estimation (\ref{algebric2}) the proof is achieved.
\end{proof}

\begin{Lemma}\label{algebric3}
	\begin{itemize}Let $0<\epsilon<1$.
		\item[(a)]\ \ If $\left(1-\epsilon\right)u\in K^{\Psi}_{\omega_{1}}(\Omega)$, then $\psi\left(\left(1-\epsilon\right)u\right)\in L^{\widetilde{\Psi}}_{\omega_{1}}(\Omega)$.
		\item[(b)]\ \ If $\left(1-\epsilon\right)u\in K^{\Psi}_{\omega_{1}}(\Omega)$ and $v\in K^{\Psi}_{\omega_{1}}(\Omega)$, then $\psi\left(\left(1-\epsilon\right)u+v\right)\in L^{\widetilde{\Psi}}_{\omega_{1}}(\Omega)$.
	\end{itemize}
\end{Lemma}

\begin{proof}
	Let $u\in K^{\Psi}_{\omega_{1}}(\Omega)$. We may assume $u,v\geq0$ a.e in $\Omega$. Then, using the monotonie of $\psi$ and inequality (\ref{messt2}), we obtain
	$$
	\omega_{1}\Psi\left(\left|u\right| \right)\geq\omega_{1}\int_{\left(1-\epsilon\right)\left|u\right|}^{\left|u\right|}\psi(s)ds\geq\omega_{1}\epsilon\left|u\right|\psi\left(\left(1-\epsilon\right)\left|u\right|\right)\geq\omega_{1}\frac{\epsilon}{1-\epsilon}\widetilde{\Psi}\left(\psi\left(\left(1-\epsilon\right) \left|u\right|\right) \right).
	$$
	It follows
	\begin{equation}\label{a}
		\int_{\Omega}\omega_{1}\Psi\left(\left|u\right| \right)\geq\frac{\epsilon}{1-\epsilon}\int_{\Omega}\omega_{1}\widetilde{\Psi}\left(\psi\left(\left(1-\epsilon\right) \left|u\right|\right) \right).
	\end{equation}
	Since $u\in K^{\Psi}_{\omega_{1}}(\Omega)$, then by (\ref{a}), we get $\psi\left(\left(1-\epsilon\right) \left|u\right|\right)\in L^{\widetilde{\Psi}}_{\omega_{1}}(\Omega)$ which proves (a). Due to the convexity of $\Psi$, we have for $0<\epsilon<1$
	
	$$
	\int_{\Omega}\omega_{1}\Psi\left(\frac{1}{1-\frac{\epsilon}{2}}\left(\left(1-\epsilon\right)u+ v\right)\right) dx\leq\frac{1-\epsilon}{1-\frac{\epsilon}{2}}\int_{\Omega}\omega_{1}\Psi\left(\left|u\right| \right)dx+\left(1-\frac{1-\epsilon}{1-\frac{\epsilon}{2}}\right)\int_{\Omega}\omega_{1}\Psi\left(\frac{2v}{\epsilon}\right)dx.
	$$
	
	Using Lemma \ref{algebric1}, we get $\frac{2v}{\epsilon}\in K^{\Psi}_{\omega_{1}}(\Omega)$, therefore $\left(1-\epsilon\right)u+v\in K^{\Psi}_{\omega_{1}}(\Omega)$. By (a), we get $\psi\left(\left(1-\epsilon\right)u+v\right)\in L^{\widetilde{\Psi}}_{\omega_{1}}(\Omega)$ which proves (b). The proof is complete.
\end{proof}

\section{Proof of Theorem \ref{T0}}

In this Section we will prove Theorem \ref{T0}. To this end we shall prove that $J$ and $I$ satisfy the conditions of Theorem \ref{bonano}.
\begin{Proposition}\label{cont}
	The functionals $I$ and $J$ are $C^{1}$.
\end{Proposition}
\begin{proof}
	The proof is similar to the proof of \cite[Lemma 4.1]{sabri}, so we omit the proof.
\end{proof}

\begin{Proposition}
	The functional $I$ is coercive (we don't need $\Delta_{2}$-condition).
\end{Proposition}
\begin{proof}
	Assume that $\left\|u\right\|_{W_{0,\omega,\omega_{1}}^{1}L^{\Phi,\Psi}\left(\Omega\right)}\geq1+\epsilon$, for some $\epsilon>0$. Then, by \eqref{es1}, we get
	$$
	\frac{\left(1+\epsilon\right)\omega(x) }{\left\|u\right\|_{W_{0,\omega,\omega_{1}}^{1}L^{\Phi,\Psi}\left(\Omega\right)}}\Phi(|\nabla u|)\geq\omega(x)\Phi\left(\frac{\left( 1+\epsilon\right) |\nabla u|}{\left\|u\right\|_{W_{0,\omega,\omega_{1}}^{1}L^{\Phi,\Psi}\left(\Omega\right)}}\right).
	$$
	It follows that
	$$
	\frac{\left(1+\epsilon\right) }{\left\|u\right\|_{W_{0,\omega,\omega_{1}}^{1}L^{\Phi,\Psi}\left(\Omega\right)}}\int_{\Omega}\omega(x)\Phi(|\nabla u|)dx\geq\int_{\Omega}\omega(x)\Phi\left(\frac{\left( 1+\epsilon\right) |\nabla u|}{\left\|u\right\|_{W_{0,\omega,\omega_{1}}^{1}L^{\Phi,\Psi}\left(\Omega\right)}}\right)dx>1.
	$$
	Letting $\epsilon$ to $0$, we obtain
	$$
	I(u)\geq\left\|u\right\|_{W_{0,\omega,\omega_{1}}^{1}L^{\Phi,\Psi}\left(\Omega\right)}.
	$$
	Hence,
	$I(u)\longrightarrow\infty$, as $\left\|u\right\|_{W_{0,\omega,\omega_{1}}^{1}L^{\Phi,\Psi}\left(\Omega\right)}\longrightarrow\infty$. The proof is complete.
\end{proof}

\begin{Proposition}\label{4.2}
	The functional $I$ is weak* lower semicontinuous (we don't need $\Delta_{2}$-condition).
\end{Proposition}

\begin{proof}
	Let $\left\lbrace u_{n} \right\rbrace_{n}\subset W_{0,\omega,\omega_{1}}^{1}L^{\Phi,\Psi}\left(\Omega\right)$ be a sequence  such that $u_{n}\rightharpoonup^{*}u$ in $W_{0,\omega,\omega_{1}}^{1}L^{\Phi,\Psi}\left(\Omega\right)$. By \cite[Lemma 1]{gossez1979orlicz}, for any given $v\in L_{\omega}^{\Phi}(\Omega)$ we have
	\begin{equation*}
		\begin{aligned}
			\sup \left\{\int_{\Omega} v \rho-\int_{\Omega} \omega\tilde{\Phi}(\rho) \mid \rho \in L_{\omega}^{\tilde{\Phi}}(\Omega)\right\} & \leq\int_{\Omega}\omega \Phi(v)
		\end{aligned}
	\end{equation*}
	and
	\begin{equation*}
		\int_{\Omega} \omega\Phi(v)\leq\sup \left\{\int_{\Omega} v \rho-\int_{\Omega} \omega\tilde{\Phi}(\rho) \mid \rho \in E_{\omega}^{\widetilde{\Phi}}(\Omega)\right\}.
	\end{equation*}
	Since $\left|\nabla u_{n} \right|\in L_{\omega}^{\Phi}(\Omega)$ for all $n\in\mathbb{N}$, we get by the above inequalities for an arbitrary positive $\epsilon$, $\exists \rho\in E_{\omega}^{\widetilde{\Phi}}(\Omega)$ such that
	\begin{equation}\label{1ST}
		\begin{aligned}
			\int_{\Omega}\omega \Phi(\left|\nabla u_{n} \right|)\geq\int_{\Omega} \left|\nabla u_{n} \right|\rho-\int_{\Omega} \omega\tilde{\Phi}(\rho)
		\end{aligned}
	\end{equation}
	and
	\begin{equation}\label{2RT}
		\int_{\Omega} \omega\Phi(\left|\nabla u \right|)\leq\int_{\Omega} \left|\nabla u \right| \rho-\int_{\Omega} \omega\tilde{\Phi}(\rho)+\epsilon.
	\end{equation}
	Clearly, we can assume that $\rho\geq0$. Combining (\ref{1ST}) with (\ref{2RT}), we obtain
	\begin{equation}\label{3RT}
		I(u_{n})-I(u)\geq \int_{\Omega} \left|\nabla u_{n}\right| \rho dx-\int_{\Omega} \left|\nabla u \right|\rho dx-\epsilon.
	\end{equation}
	On the other hand we have, $\nabla u_{n} \rho\rightharpoonup^{*}\nabla u\rho$ in $L_{\omega}^{1}(\Omega)$ and by the weak lower semicontinuity of norms we get
	\begin{equation}\label{4RT}
		\int_{\Omega}\omega|\nabla u| \rho=\|\nabla u \rho\|_{L_{\omega}^{1}(\Omega)} \leq \liminf \left\|\nabla u_n \rho\right\|_{L_{\omega}^{1}(\Omega)}=\liminf \int_{\Omega}\omega\left|\nabla u_n\right| \rho.
	\end{equation}
	Finaly, conbining (\ref{3RT}) with (\ref{4RT}) we get
	\begin{equation*}
		\liminf I\left(u_n\right) \geq I(u)-\epsilon.
	\end{equation*}
	Since $\epsilon$ is arbitrary, the proof is complete.
\end{proof}
\begin{Proposition}\label{homeomor}
	Under $(\Phi_{1})$, $I$ is convex, sequentially weakly lower semi-continuous and $I':W_{0,\omega,\omega_{1}}^{1}L^{\Phi,\Psi}\left(\Omega\right)\rightarrow \left(W_{0,\omega,\omega_{1}}^{1}L^{\Phi,\Psi}\left(\Omega\right) \right)'$ is an homeomorphism.
\end{Proposition}
\begin{proof}
	Due to the convexity of $\Phi$ one has the convexity of $I$ and
	from Lemma \ref{4.2}, if $\Phi$ satisfies $\Delta_{2}$ condition, then $I$ sequentially weakly lower semi-continuous. Now, we claim that $I':X\rightarrow X^*$ is an homeomorphism. Indeed, taking into account that $I^{\prime}$ is bounded, It remains to be showed that if $\{u_n\}_{n\in\mathbb{N}}$ is a sequence in $W_{0,\omega,\omega_{1}}^{1}L^{\Phi,\Psi}\left(\Omega\right)$ such that
	\begin{equation} \label{asump}
		u_n\rightharpoonup u, \quad I'(u_n)\rightharpoonup v, \quad \langle I'(u_n),u_n\rangle\to\langle v,u\rangle\
	\end{equation}
	then $u_n \to u$ in $W_{0,\omega,\omega_{1}}^{1}L^{\Phi,\Psi}\left(\Omega\right)$.

	Since $\Phi$ is convex, we have
	$$
	\Phi(|\nabla u|)\leq \Phi\bigg{(}\bigg{|}\frac{\nabla u+\nabla u_n}{2}\bigg{|}\bigg{)}+\varphi(|\nabla u|)\frac{\nabla u}{|\nabla u|}.\frac{\nabla u-\nabla u_n}{2}
	$$
	and
	$$
	\Phi(|\nabla u_{n}|)\leq \Phi\bigg{(}\bigg{|}\frac{\nabla u+\nabla u_n}{2}\bigg{|}\bigg{)}+\varphi(|\nabla u_n|)\frac{\nabla u_n}{|\nabla u_n|}.\frac{\nabla u_n-\nabla u}{2}
	$$
	Adding the above two relations and integrating over $\Omega$ we find that
	\begin{align*}\label{umo1}
		\frac{1}{2}&\int_{\Omega}\omega(x) \Big(\varphi(|\nabla u|)\frac{\nabla u}{|\nabla u|}-\varphi(|\nabla u_n|)\frac{\nabla u_n}{|\nabla u_n|}  \Big)(\nabla u-\nabla u_n)\ dx\nonumber\\
		&\geq \int_{\Omega}\omega(x)\Phi\left(\left|\nabla u\right| \right)dx  + \int_{\Omega}\omega(x)\Phi\left(\left|\nabla u_{n}\right| \right)dx - 2 \int_{\Omega}\omega(x)\Phi\left(\frac{\nabla u+\nabla u_n}{2} \right)dx.
	\end{align*}
	By applying Lemma \ref{important2}, we get
	$$
	\langle I'(u)-I'(u_j), u-u_n\rangle \geq 4 \int_{\Omega}\omega(x)\Phi\left(\nabla u-\nabla u_{n}\right)dx.
	$$
	This, together  with Proposition \ref{important} yields
	\begin{equation} \label{eqxx1}
		\langle I'(u)-I'(u_n), u-u_j\rangle \geq 4\min\left\lbrace \left\|u-u_{n}\right\|^{l}_{W_{0,\omega,\omega_{1}}^{1}L^{\Phi,\Psi}\left(\Omega\right)},\left\|u-u_{n}\right\|^{m}_{W_{0,\omega,\omega_{1}}^{1}L^{\Phi,\Psi}\left(\Omega\right)}\right\rbrace .
	\end{equation}
	
	On the other hand, the compact embedding $W_{0,\omega,\omega_{1}}^{1}L^{\Phi,\Psi}\left(\Omega\right)\hookrightarrow\hookrightarrow L_{\omega_{1}}^{\Psi}\left(\Omega\right)$ gives that $u_n\to u$ in $L_{\omega_{1}}^{\Psi}\left(\Omega\right)$ and a.e. in $\Omega$, which, mixed up with the assumptions \eqref{asump} allows us to deduce that
	$$
	\displaystyle\lim_{n\to\infty}\langle I'(u_n)-I'(u),u_n-u \rangle=\displaystyle\lim_{n\to\infty} \left(\langle I'(u_n),u_n\rangle-\langle I'(u_n),u\rangle-\langle I'(u),u_n-u\rangle \right)=0.
	$$
	Hence, from \eqref{eqxx1}, $\left\|u-u_{n}\right\|_{W_{0,\omega,\omega_{1}}^{1}L^{\Phi,\Psi}\left(\Omega\right)}\to0$ as $n\to\infty$.
	This allow us to claim that $I$ is of type (S+). In addition since $I$ is strictly monotone, $I$ is an injection. Since $I$ is coercive, thus $I'$ is a surjection. Hence, $I'$ has an inverse mapping $\left(I'\right)^{-1}:\left( W_{0,\omega,\omega_{1}}^{1}L^{\Phi,\Psi}\left(\Omega\right)\right)'\mapsto W_{0,\omega,\omega_{1}}^{1}L^{\Phi,\Psi}\left(\Omega\right)$.
	If $f_n, f \in \left( W_{0,\omega,\omega_{1}}^{1}L^{\Phi,\Psi}\left(\Omega\right)\right)', f_n \rightarrow f$, let $u_n=\left(I^{\prime}\right)^{-1}\left(f_n\right), u=\left(I^{\prime}\right)^{-1}(f)$, then $I^{\prime}\left(u_n\right)=f_n, I^{\prime}(u)=f$. So, $\left\{u_n\right\}$ is bounded in $W_{0,\omega,\omega_{1}}^{1}L^{\Phi,\Psi}\left(\Omega\right)$. Since $f_n \rightarrow f$, we get
	$$
	\lim _{n \rightarrow \infty}\langle I^{\prime}\left(u_n\right)-I^{\prime}\left(u_0\right), u_n-u\rangle=\lim _{n \rightarrow \infty}\langle f_n, u_n-u\rangle=0 .
	$$
	Since $I$ is of type $\left(S_{+}\right)$, we conclude that $u_{n} \rightarrow u$, hence, $I^{\prime}$ is continuous. The proof is complete.
\end{proof}
\begin{Lemma}
	$J'$ is compact
\end{Lemma}
\begin{proof}
	Using similar proof as in \cite[Lemma 4.1]{sabri}, we get that $J'$ is completly continous and by Proposition \ref{compact} the proof is complete.
\end{proof}
Now, we are in position to prove Theorem \ref{T0}
\begin{proof}[\textbf{Proof of Theorem \ref{T0}}]
	\begin{equation*}
		v_{d}:=\left\{
		\begin{array}{ccc}
			&0,   & \; \; \; \; \; \mbox{if}\ \  x\in\Omega\setminus B(x_0,D),\\
			&  \displaystyle \frac{2d}{D}(D-|x-x_0|),  &\qquad\; \; \; \; \; \; \; \; \mbox{if}\ \ x\in B(x_0,D)\setminus B(x_0,\frac{D}{2}),\\
			&d,  & \mbox{if}\ \ x\in B(x_0,\frac{D}{2}),\\
		\end{array}
		\right.
	\end{equation*}	
	it is easy to see that
	\begin{equation*}
		\frac{\partial v_{d}(x)}{\partial x_{i}}:=\left\{
		\begin{array}{ccc}
			&0,   & \; \; \; \; \; \mbox{if}\ \  x\in \Omega \backslash B\left(x_{0}, D\right) \cup B\left(x_{0}, \frac{D}{2}\right),\\
			&  \displaystyle \frac{-2d}{D}\frac{\left(x_i-x_{i,0} \right) }{\left|x_i-x_0 \right| },  &\qquad\; \; \; \; \; \; \; \; \mbox{if}\ \ x\in B(x_0,D)\setminus B(x_0,\frac{D}{2}),\\
			
		\end{array}
		\right.
	\end{equation*}
	we have
	\begin{equation*}
		\begin{aligned}
			I(v_{d})&=\int_{\Omega}\omega(x)\Phi\left(\left|\nabla v_{d}\right|\right)dx\\
			&=\int_{B(x_0,D)\setminus B(x_0,\frac{D}{2})}\omega(x)\Phi\left(\frac{2d}{D}\right)dx
		\end{aligned}
	\end{equation*}
	By Proposition \ref{important}, we get
	\begin{equation}\label{ine}
		\begin{aligned}
			\min\left\lbrace\left\|\frac{2d}{D}\right\|_{L_{\omega}^{\Phi}\left(\Omega\right)}^{l},\left\|\frac{2d}{D}\right\|_{L_{\omega}^{\Phi}\left(\Omega\right)}^{m}  \right\rbrace\leq I(v_{d})\leq\max\left\lbrace\left\|\frac{2d}{D}\right\|_{L_{\omega}^{\Phi}\left(\Omega\right)}^{l},\left\|\frac{2d}{D}\right\|_{L_{\omega}^{\Phi}\left(\Omega\right)}^{m} \right\rbrace \\
		\end{aligned}
	\end{equation}
	On the other hand,
	\begin{equation*}
		\begin{aligned}
			J(v_{d})&=\int_{\Omega}\omega_{1}(x)\Psi\left(\left| v_{d}\right|\right)dx\\
			&\geq\int_{B(x_0,\frac{D}{2})}\omega_{1}(x)\Psi\left(\left|d\right| \right)dx\\
			&\geq\min\left\lbrace\left|d\right|^{l_{1}},\left|d\right|^{m_{1}}\right\rbrace\Psi\left(1\right)\frac{\pi^{\frac{N}{2}}}{\frac{N}{2}\Gamma(\frac{N}{2})}\left(\frac{D}{2}\right)^{N}
		\end{aligned}
	\end{equation*}
	Hence,
	\begin{equation*}
		\begin{aligned}
			\frac{J(v_{d})}{I(v_{d})}\geq\frac{\min\left\lbrace\left|d\right|^{l_{1}},\left|d\right|^{m_{1}}\right\rbrace\Psi\left(1\right)\frac{\pi^{\frac{N}{2}}}{\frac{N}{2}\Gamma(\frac{N}{2})}\left(\frac{D}{2}\right)^{N} }{\left(2N\right)^{m}\max\left\lbrace\left\|\frac{d}{D}\right\|_{L_{\omega}^{\Phi}\left(\Omega\right)}^{l},\left\|\frac{d}{D}\right\|_{L_{\omega}^{\Phi}\left(\Omega\right)}^{m} \right\rbrace}:=\gamma_{d}.
		\end{aligned}
	\end{equation*}
	By (\ref{ine}) and since $r<\min\left\lbrace\left\|\frac{2Nd}{D}\right\|_{L_{\omega}^{\Phi}\left(\Omega\right)}^{l},\left\|\frac{2Nd}{D}\right\|_{L_{\omega}^{\Phi}\left(\Omega\right)}^{m}  \right\rbrace$, we obtain that $I(v_{d})>r$. In the sequel, let $u\in$ such that $I(u)<r$. Notice that by Proposition \ref{important}, we have
	\begin{equation*}
		r>I(u):=\int_{\Omega}\omega(x)\Phi(\left|\nabla u\right|)dx\geq\min\left\lbrace\left\|\nabla u\right\|_{L_{\omega}^{\Phi}\left(\Omega\right)}^{l},\left\|\nabla u\right\|_{L_{\omega}^{\Phi}\left(\Omega\right)}^{m}  \right\rbrace.
	\end{equation*}
	Using the above inequality with Proposition \ref{poincare}, we get
	\begin{equation*}
		\begin{aligned}
			J(u)&=\int_{\Omega}\omega_{1}(x)\Psi\left(\left| u\right|\right)dx\\
			&\leq\max\left\lbrace\left\| u\right\|_{L_{\omega_{1}}^{\Psi}\left(\Omega\right)}^{l_{1}},\left\| u\right\|_{L_{\omega_{1}}^{\Psi}\left(\Omega\right)}^{m_{1}}  \right\rbrace\\
			&\leq\max\left\lbrace C^{l_{1}}_{1},C^{m_{1}}_{1} \right\rbrace\max\left\lbrace\left\| \nabla u\right\|_{L_{\omega}^{\Phi}\left(\Omega\right)}^{l_{1}},\left\| \nabla u\right\|_{L_{\omega}^{\Phi}\left(\Omega\right)}^{m_{1}}  \right\rbrace\\
			&\leq\max\left\lbrace C^{l_{1}}_{1},C^{m_{1}}_{1} \right\rbrace\max\left\lbrace \max\left\lbrace\left| r \right|^{\frac{l_{1}}{l}},\left| r\right|^{\frac{l_{1}}{m}}\right\rbrace,\max\left\lbrace\left| r \right|^{\frac{m_{1}}{l}},\left| r\right|^{\frac{m_{1}}{m}}  \right\rbrace\right\rbrace .
		\end{aligned}
	\end{equation*}
	Therefore
	\begin{equation*}
		\frac{1}{r}\underset{I(u)<r}{\sup}J(u)\leq\max\left\lbrace C^{l_{1}}_{1},C^{m_{1}}_{1} \right\rbrace\max\left\lbrace \max\left\lbrace\left| r \right|^{\frac{l_{1}}{l}},\left| r\right|^{\frac{l_{1}}{m}}\right\rbrace,\max\left\lbrace\left| r \right|^{\frac{m_{1}}{l}},\left| r\right|^{\frac{m_{1}}{m}}  \right\rbrace\right\rbrace:=\tilde{\omega_{r}}
	\end{equation*}
	In the next step, we shall prove that for each $\alpha>0$, the energy functional $I-\lambda J$ is coercive. Indeed, take $u\in W_{0,\omega,\omega_{1}}^{1}L^{\Phi,\Psi}\left(\Omega\right)$ such that 
	$\left\|u \right\|_{W_{0,\omega,\omega_{1}}^{1}L^{\Phi,\Psi}\left(\Omega\right)}=\left\|\nabla u \right\|_{L_{\omega}^{\Phi}(\Omega)}\geq1$.
	By (\ref{*}) in Lemma \ref{algebric1}, we have
	$$
	\begin{aligned}
		\int_{\Omega}\omega_{1}(x)\Psi( u(x))dx
		\leq2\Psi(T)\left\|\omega_{1}\right\|_{L_{\omega_{1}}^{\widetilde{\Psi}}\left(\Omega_{k}\right)}\left\|1\right\|_{L_{\omega_{1}}^{\Psi}\left(\Omega_{k}\right)}
		+1.
	\end{aligned}
	$$
	It follows from Proposition \ref{important} and using the fact that $\left\|\nabla u \right\|_{L_{\omega}^{\Phi}(\Omega)}\geq1$ yields
	\begin{equation*}
		\begin{aligned}
			I(u)-\lambda J(u)&\geq\int_{\Omega}\omega(x)\Phi\left(\left|\nabla u\right|\right)dx-\lambda\left[ 2\Psi(T)\left\|\omega_{1}\right\|_{L_{\omega_{1}}^{\widetilde{\Psi}}\left(\Omega_{k}\right)}\left\|1\right\|_{L_{\omega_{1}}^{\Psi}\left(\Omega_{k}\right)}
			+1\right]\\
			&\geq\min\left\lbrace\left\| u\right\|_{L_{\omega_{1}}^{\Psi}\left(\Omega\right)}^{l_{1}},\left\| u\right\|_{L_{\omega_{1}}^{\Psi}\left(\Omega\right)}^{m_{1}}  \right\rbrace-\lambda\left[ 2\Psi(T)\left\|\omega_{1}\right\|_{L_{\omega_{1}}^{\widetilde{\Psi}}\left(\Omega_{k}\right)}\left\|1\right\|_{L_{\omega_{1}}^{\Psi}\left(\Omega_{k}\right)}
			+1\right]\\
			&\geq\left\| \nabla u\right\|_{L_{\omega}^{\Phi}\left(\Omega\right)}^{l}-\lambda\left[ 2\Psi(T)\left\|\omega_{1}\right\|_{L_{\omega_{1}}^{\widetilde{\Psi}}\left(\Omega_{k}\right)}\left\|1\right\|_{L_{\omega_{1}}^{\Psi}\left(\Omega_{k}\right)}
			+1\right]
		\end{aligned}
	\end{equation*}
	Thus, $I(u)-\lambda J(u)$ coercive for all $\lambda>0$, particulary for $\lambda\in\Lambda_r:=\left( \frac{I(v_{d})}{J(v_{d})}, \frac{r}{\displaystyle\sup_{I(u)\leq r}J(u)}\right)$.
	It is clear that $I$ and $J$ satisfy Theorem \ref{bonano} conditions. Therefore, for each $\lambda\in\Lambda_r$, the functional $I-\lambda J$ has at least three distinct critical points in $W_{0,\omega,\omega_{1}}^{1}L^{\Phi,\Psi}\left(\Omega\right)$. The proof is complete.
\end{proof}

\section{Proof of Theorem \ref{T05}} \label{LS}
We consider the space $X=W_{0,\omega,\omega_{1}}^{1}L^{\Phi,\Psi}\left(\Omega\right)$ and the functionals $B(u)=I(u)$, $A(u)=J(u)$, where $X$, $A$ and $B$ are mentioned in Theorem \ref{B1}.
\begin{proof}[Proof of Theorem \ref{T05}]
	By Proposition \ref{cont} the functionals $I$ and $J$ are  $C^{1}\left( W_{0,\omega,\omega_{1}}^{1}L^{\Phi,\Psi}\left(\Omega\right),\R\right)$ and it is clear that $I(0)=J(0)=0$. Thus, the condition $(h_{1})$ of Theorem \ref{B1} hold true.  We notice that from assumption $(\Psi_{1})$ and Proposition \ref{important}
	$$
	l \xi^-(\|u\|_{L_{\omega_{1}}^\Psi(\Omega)})\leq l J(u)\leq \langle J'(u),u\rangle \leq m J(u)\leq m\xi^+(\|u\|_{L_{\omega_{1}}^\Psi(\Omega)}),
	$$
	where, $\xi^-$ and $\xi^+$ denote $\min\left\lbrace t^{l},t^{m} \right\rbrace$ and $\max\left\lbrace t^{l},t^{m} \right\rbrace$ respectively.
	Then, immediately we obtain
	$$
	\langle J'(u),u\rangle=0 \ \Leftrightarrow \ J(u)=0 \ \Leftrightarrow \ u=0.
	$$
	So, it is still necessary to verify that $J'$ is strongly continuous. Let $u_n\rightharpoonup u$ in $W_{0,\omega,\omega_{1}}^{1}L^{\Phi,\Psi}\left(\Omega\right)$, then $\{u_n\}_{n\in\mathbb{N}}$ is bounded in $W_{0,\omega,\omega_{1}}^{1}L^{\Phi,\Psi}\left(\Omega\right)$. We must demonstrate that $J'(u_n)\to J'(u)$ in $(W_{0,\omega,\omega_{1}}^{1}L^{\Phi,\Psi}\left(\Omega\right))^{'}$. Notice that
	\begin{align*}
		|\langle J'(u_{n})-J'(u),v\rangle|&=\left|\int_{\Omega}\left(\omega_{1}\psi(|u_n|)\frac{u_n}{|u_{n}|}-\omega_{1}\psi(|u|)\frac{u}{|u|}\right) v\,dx\right|\\
		&\leq\left|\int_{\Omega} \omega_{1}\psi(|u_n|)\left(\frac{u_{n}}{|u_{n}|}-\frac{u}{|u|}\right) v\,dx\right| +\left|\int_{\Omega}\omega_{1}(\psi(|u_n|)-\psi(|u|))\frac{u}{|u|}v\,dx\right| \\
		&:=J_{1,n}+J_{2,n}.
	\end{align*}
	
	We indicate $\frac{u_{n}}{|u_{n}|}-\frac{u}{|u|}$ by $X_n$. Next, we show that 
	\begin{equation}\label{J1N}
		J_{1,n}\rightarrow0 \ \text{as}\  n\rightarrow\infty.
	\end{equation}
	Indeed, Lemma \ref{important3} and H\"{o}lder's inequality may be used to get
	$$
	J_{1,n}\leq \|\psi(|u_n|)\|_{L_{\omega_{1}}^{\widetilde \Psi}(\Omega)}\|X_n v\|_{L_{\omega_{1}}^\Psi(\Omega)}\rightarrow 0,\ n\rightarrow+\infty.
	$$
	On the other hand since $u_n\to u$ a.e. in $\Omega$ and $\Psi(|X_n v|)\leq 2\Psi(|v|)\in\ L^1(\Omega)$, we deduce that
	$$
	\Psi(|X_n v|)\to 0\ \text{a.e. in}\ \Omega.
	$$
	By applying dominated convergence theorem, we infer that
	$$
	\|X_n v\|_{L_{\omega_{1}}^\Psi(\Omega)}\rightarrow 0,\ n\rightarrow+\infty.
	$$
	Moreover, $(\psi(|u_n|))$ is bounded in $L_{\omega_{1}}^{\widetilde{\Psi}}(\Omega)$. So
	$I_{1,n}\rightarrow 0,\ n\rightarrow+\infty$ as expected.
	
	Similarly, by using the $\Delta_2$ condition, Lemma \ref{important3} and H\"{o}lder's inequality we get
	$$
	J_{2,n}\leq \|\psi(|u_n|)-\psi(|u|)\|_{L_{\omega_{1}}^{\widetilde \Psi}(\Omega)}\|v\|_{L_{\omega_{1}}^\Psi(\Omega)}.
	$$
	We need to prove that 
	\begin{equation}\label{J2N}
		J_{2,n}\rightarrow0 \ \text{as}\  n\rightarrow\infty.
	\end{equation}. Indeed, since $u_n \rightharpoonup u$ in $W_{0,\omega,\omega_{1}}^{1}L^{\Phi,\Psi}\left(\Omega\right)$, in light of the compact embedding $W_{0,\omega,\omega_{1}}^{1}L^{\Phi,\Psi}\left(\Omega\right)\hookrightarrow\hookrightarrow L_{\omega_{1}}^{\Psi}\left(\Omega\right)$, $u_n\to u$ strongly in $L_{\omega_{1}}^\Psi(\Omega)$ and a.e. in $\R^n$. Moreover, $\widetilde{\Psi}(|\psi(|u_n|)-\psi(|u|)|)\to 0$  a.e. in $\Omega$. On the other hand, by Lemma \ref{lema.prop} and Lemma \ref{important3}, we have
	\begin{align*}
		\tilde{\Psi}(|\psi(|u_n|)-\psi(|u|)|)& \leq \widetilde{\Psi}(|\psi(|u_n|)+\psi(|u|)|)\\
		&\leq \mathbf{\widetilde C} \left[\widetilde{\Psi}(|\psi(|u_n|)|)+\widetilde{\Psi}(|\psi(|u|)|)\right]\\
		&\leq \mathbf{\widetilde C}m \left[ \Psi(|u_n|)+ \Psi(|u|)\right]\\
		&\leq \mathbf{\widetilde C}m \left[K+ \Psi(|u|)\right]
	\end{align*}where $K>0$ is such that $|\Psi(|u_n|)|\leq K$.\\
	Thus, by dominated convergence theorem, $\int_\Omega\omega_{1} \widetilde{\Psi}(|\psi(|u_n|)-\psi(|u|)|)\,dx \to 0$ and hence $\|\psi(|u_n|)-\psi(|u|)\|_{L_{\omega_{1}}^{\widetilde{\Psi}}(\Omega)} \to 0$ as $n\to\infty$.
	Therefore, From (\ref{J1N}) and (\ref{J2N}) one can deduce that $\|J'(u_{n})-J'(u)\|_{(W_{0,\omega,\omega_{1}}^{1}L^{\Phi,\Psi}\left(\Omega\right))^{'}}\rightarrow0$ as required. Hence, the condition ($h_{3}$) of Theorem \ref{B1} hold true. Next, by the proof of the Proposition \ref{homeomor} the condition ($h_{2}$) verified. Moreover, for any $u\in W_{0,\omega,\omega_{1}}^{1}L^{\Phi,\Psi}\left(\Omega\right)\setminus\{0\}$,
	$$
	\langle I'(u),u\rangle>0,\qquad \displaystyle\lim_{t\to+\infty}I(tu)=+\infty,\qquad \displaystyle\inf_{u\in M_\alpha}\langle I'(u),u\rangle>0.
	$$
	Hence, there exist a sequence of positive numbers $\{\mu_{k,\alpha}\}_{k\in\mathbb{N}}$ tending to 0 and a corresponding sequence of functions $\{u_{k,\alpha}\}_{k\in\mathbb{N}}\in W_{0,\omega,\omega_{1}}^{1}L^{\Phi,\Psi}\left(\Omega\right)$ such that
	$$
	\left\langle-\operatorname{div}\left(\omega(x)\frac{\varphi\left(\nabla u_{k,\alpha}\right)}{\left|\nabla u_{k,\alpha} \right| }\nabla u_{k,\alpha}\right), v\right\rangle=\frac{1}{\mu_{k,\alpha}}\int_{\Omega} \omega_{1}(x)\psi(|u_{k,\alpha}|)\frac{u_{k,\alpha}}{|u_{k,\alpha}|}v\, dx,
	$$
	$\forall v\in W_{0,\omega,\omega_{1}}^{1}L^{\Phi,\Psi}\left(\Omega\right)$.
	Moreover, $J(u_{k,\alpha})=\alpha$ and
	$$
	I(u_{k,\alpha}):=c_{k,\alpha}=\sup_{K\in \mathcal{C}_k}\inf_{u\in K}I(u)>0.
	$$
	Consequently, $\lambda_{k,\alpha}=1/\mu_{k,\alpha}$ is an eigenvalue of (\ref{mainproblem0}) with eigenfunction $u_{k,\alpha}$. The proof is complete.
\end{proof}
\begin{Remark}\label{remark}
	Since $J(u_{k,\alpha})=\alpha$, we obtain
	\begin{equation*}
		\int_{\Omega}\omega_{1}\Psi(\left|u_{k,\alpha}\right|)dx<\infty.
	\end{equation*}
	Thus, $u_{k,\alpha}\in L_{\omega_{1}}^{\Psi}(\Omega)$. Hence, The Problem (\ref{mainproblem0}) can be studied with an analogue way in $W^{1}_{0,\omega}L^{\Phi}(\Omega)$.
\end{Remark}
\section{Eigenvalues Problem without $\Delta_{2}$}
In this part of paper, neither $\Phi$ nor its conjugate function $\tilde{\Phi}$ satisfy the $\Delta_{2}$-condition and we assum that $\Psi\prec\prec\Phi$. Then, $I$ and $J$ are not $C^{1}$ in general.
\begin{Lemma}\label{weakclosed}
	The set $M_{\alpha}$ is sequentially weak* closed.
\end{Lemma}

\begin{proof}
	Let $\left\lbrace u_{n} \right\rbrace_{k}\subset M_{\alpha}$ be a sequence such that $u_{n}\rightharpoonup^*u$. We show that $J(u)=\alpha$. Since $\psi$ is increasing and by the H$\ddot{o}$lder inequality (\ref{weithedh}), we obtain
	\begin{equation*}
		\begin{aligned}
			\left|J(u_{n})-J(u)\right|&\leq\int_{\Omega}\omega_{1}(x)\left|\Psi\left(\left|u_{n}(x)\right|\right)-\Psi\left(\left|u(x)\right|\right)\right|dx\\
			&\leq\int_{\Omega}\omega_{1}(x)\left(\int_{\left|u(x)\right|}^{\left|u(x)\right|+\left|u_{n}(x)-u(x)\right|}\psi(s)ds\right)dx\\
			&\leq\int_{\Omega}\omega_{1}(x)\psi\left(\left|u(x)\right|+\left|u_{n}(x)-u(x)\right|\right)\left|u_{n}(x)-u(x)\right|dx\\
			&\leq2\left\|\psi\left(\left|u\right|+\left|u_{n}-u\right|\right) \right\|_{L_{\omega_{1}}^{\widetilde{\Psi}}\left(\Omega\right)} \left\|u_{n}-u \right\|_{L_{\omega_{1}}^{\Psi}\left(\Omega\right) }.
		\end{aligned}
	\end{equation*}
	With this estimation we just need to prove that $\psi\left(\widetilde{u}_{n}\right)$ is uniformly bounded in $L_{\omega_{1}}^{\widetilde{\Psi}}\left(\Omega\right)$, where $\tilde{u}_{n}(x):=\left|u(x)\right|+\left|u_{n}(x)-u(x)\right|$. Indeed, if $\left\|\psi\left(\widetilde{u}_{n}\right) \right\|_{L_{\omega_{1}}^{\widetilde{\Psi}}\left(\Omega\right)}\leq2$, then the proof is acheived. To this end, assume that $\left\|\psi\left(\widetilde{u}_{n}\right) \right\|_{L_{\omega_{1}}^{\widetilde{\Psi}}\left(\Omega\right)}>2$, then since $\widetilde{\Psi}$ is covexe, we get
	\begin{equation*}
		1<\int_{\Omega}\omega_{1}(x)\widetilde{\Psi}\left(\frac{2\psi\left(\widetilde{u}_{n}\right)}{\left\|\psi\left(\widetilde{u}_{n}\right) \right\|_{L_{\omega_{1}}^{\widetilde{\Psi}}\left(\Omega\right)}} \right)dx\leq\frac{2}{\left\|\varphi\left(\widetilde{u}_{n}\right) \right\|_{L_{\omega_{1}}^{\widetilde{\Psi}}\left(\Omega\right)}}\int_{\Omega}\omega_{1}(x)\widetilde{\Psi}\left(\psi\left(\widetilde{u}_{n}\right)\right)dx.
	\end{equation*}
	It follows from (\ref{2.2}) and (\ref{messt2})
	\begin{equation*}
		\begin{aligned}
			\frac{\left\|\psi\left(\widetilde{u}_{n}\right) \right\|_{L_{\omega_{1}}^{\widetilde{\Psi}}\left(\Omega\right)}}{2}&\leq\int_{\Omega}\omega_{1}(x)\widetilde{\Psi}\left(\psi\left(\widetilde{u}_{n}\right)\right)dx\\
			&\leq\int_{\Omega}\omega_{1}(x)\left(\int_{0}^{\phi\left(\widetilde{u}_{n}\right)}\psi^{-1}(s)ds\right) dx\\
			&\leq\int_{\Omega}\omega_{1}(x)\widetilde{u}_{n}\psi\left(\widetilde{u}_{n}\right)dx\\
			&\leq\int_{\Omega}\omega_{1}(x)\Psi\left(2\widetilde{u}_{n}\right) dx.
		\end{aligned}
	\end{equation*}
	Moreover, since $\Psi$ is convexe we have for all $\epsilon>0$
	\begin{small}
		\begin{equation*}
			\int_{\Omega}\omega_{1}(x)\Psi\left(2\widetilde{u}_{n}\right) dx\leq\frac{1}{1+\epsilon}\int_{\Omega}\omega_{1}(x)\Psi\left(2\left(1+\epsilon\right) \left|u\right|\right)dx+\frac{\epsilon}{1+\epsilon}\int_{\Omega}\omega_{1}(x)\Psi\left(2\frac{1+\epsilon}{\epsilon}\left|u_{n}-u\right|\right)dx.
		\end{equation*}
	\end{small}By Lemma \ref{algebric1}, we have $2\left(1+\epsilon\right) \left|u\right|\in K_{\omega}^{\Psi}(\Omega)$ and $2\frac{(1+\epsilon)}{\epsilon}\left|u_{n}-u\right|\in K_{\omega}^{\Psi}(\Omega)$, then $\psi\left(\widetilde{u}_{n}\right)$ is uniformly bounded in $L_{\omega_{1}}^{\widetilde{\Psi}}\left(\Omega\right)$. Thus, $I(u_{n})\longrightarrow I(u)$. Hence, $I(u)=\alpha$ and $u\in M_{\alpha}$.
	The proof is complete.
\end{proof}

Now we are in position to prove Theorem \ref{T1}.
\begin{proof}[\textbf{Proof of Theorem \ref{T1}}]
	We assert that $M_{\alpha}$ is not empty since $\Psi$ is increasing. Given $\left\lbrace u_{n}\right\rbrace_{n}\subset M_{\alpha}$, a minimizing sequence for (\ref{minid}), i.e
	\begin{equation*}
		\underset{n\longrightarrow\infty}{\lim}I\left(u_{n}\right)=\Lambda_{\alpha}.
	\end{equation*}
	Since $I$ is coercive, then $\left\lbrace u_{n}\right\rbrace_{n}$ is bounded in $W_{0,\omega,\omega_{1}}^{1}L^{\Phi,\Psi}\left(\Omega\right)$ which is the dual of a separable Banach space. Thus, up to a subsequence still denoted by $\left\lbrace u_{n}\right\rbrace_{n}$, there exists $u\in W_{0,\omega,\omega_{1}}^{1}L^{\Phi,\Psi}\left(\Omega\right)$ stisfies $u_{n}\rightharpoonup^{*}u$ in $W_{0,\omega,\omega_{1}}^{1}L^{\Phi,\Psi}\left(\Omega\right)$. By Lemma \ref{weakclosed}, we obtain that $u\in M_{\alpha}$ and since $I$ is weak* lower semicontinuous, we get
	\begin{equation}\label{zb}
		I(u)\leq\underset{n\longrightarrow\infty}{\liminf}I(u_{n})=\Lambda_{\alpha}^{\left(\ref{mainproblem0}\right)}.
	\end{equation}
	Finaly, by (\ref{zb}) and the definition of $\Lambda_{\alpha}^{\left(\ref{mainproblem0}\right)}$, we find
	\begin{equation*}\label{2}
		I(u)=\Lambda_{\alpha}^{\left(\ref{mainproblem0}\right)}.
	\end{equation*}
	Hence, $u$ is a solution of minimization problem (\ref{minid}) and can be assumed to be one-signed in $\Omega$, since if u solves
	minimization problem (\ref{minid}), then $\left|u\right|$ do.	
\end{proof}
\begin{Lemma}\label{delta}
	Given $u,v\in E_{\omega_{1}}^{\Psi}(\Omega)$, such that $u\neq0$ and $\int_{\Omega}\omega_{1}(x)\psi\left(\left|u\right|\right)vdx\neq0$. Then
	\begin{equation*}
		\int_{\Omega}\omega_{1}\Psi\left(\left(1-\epsilon\right)\left| u\right| +\delta \left|v\right| \right)dx=\int_{\Omega}\omega_{1}\Psi\left(\left|u\right|\right)
	\end{equation*}
	defines continuously differentiate function $\delta$ of $\epsilon$ in some interval $\left(-\epsilon_{0},\epsilon_{0}\right)$ with $\epsilon_{0}>0$, such that $\delta(0)=0$ and
	\begin{equation*}
		\delta^{'}(0)=\frac{\int_{\Omega}\omega_{1}\psi\left(\left|u \right|\right)udx}{\int_{\Omega}\omega_{1}\psi\left(\left|u \right| \right)vdx}.
	\end{equation*}
\end{Lemma}
\begin{proof}[\textbf{Proof}]
	Denote $D=(-1,1) \times(-1,1) \subset \mathbb{R}^2$ and define $g: D \rightarrow \mathbb{R}$ by
	$$
	g(\epsilon, \delta)=\int_{\Omega}\omega_{1} \Psi((1-\epsilon) \left|u \right|+\delta v)-\int_{\Omega} \omega_{1}\Psi(\left|u \right|) .
	$$
	By Lemma (\ref{algebric1}), since $2|u|+|v|\in E_{\omega_{1}}^{\Psi}(\Omega)$, then $2|u|+|v|\in K_{\omega_{1}}^{\Psi}(\Omega)$. Thus,
	$$
	\omega_{1}\Psi((1-\epsilon) \left|u \right|+\delta \left|v\right| ) \leq \omega_{1}\Psi(2|u|+|v|) \in L^1(\Omega) \text { for all }(\epsilon, \delta) \in D.
	$$
	Therefore, the dominated convergence theorem implies that $g$ is continous in $D$. Moreover, for all $(\epsilon, \delta) \in D$
	\begin{equation}\label{partialepsi}
		\left|\frac{\partial}{\partial \epsilon} \Psi\left((1-\epsilon) \left| u\right| +\delta \left|v\right| \right)\right|=\left|-\psi((1-\epsilon) u+\delta v) u\right| \leq \psi\left(2|u|+|v|\right)|u|
	\end{equation}
	By inequality (\ref{messt}), we obtain that
	\begin{equation*}
		\omega_{1}\psi(2|u|+|v|)|u|\leq \omega_{1}\psi\left(2|u|+|v|\right)\left(2|u|+|v|\right)\leq\omega_{1}\Psi\left(4|u|+2|v|\right).
	\end{equation*}
	Hence, by Lemma (\ref{algebric1}), that $4|u|+2|v|\in K_{\omega_{1}}^{\Psi}(\Omega)$. Thus, by (\ref{partialepsi}) we have
	\begin{equation*}
		\left|\omega_{1}\frac{\partial}{\partial \epsilon} \Psi\left((1-\epsilon) \left| u\right| +\delta \left| v\right| \right)\right|=\left|-\omega_{1}\psi((1-\epsilon) \left| u\right| +\delta \left| v\right| ) u\right| \leq \omega_{1}\psi\left(2|u|+|v|\right)|u|\in L^1(\Omega).
	\end{equation*}
	Similarly, we get for all $(\epsilon, \delta) \in D$
	\begin{equation*}
		\left|\omega_{1}\frac{\partial}{\partial \delta} \Psi((1-\epsilon) \left| u\right| +\delta \left| v\right| )\right|=|\omega_{1}\psi((1-\epsilon) u+\delta v) v| \leq \omega_{1}\psi(2|u|+|v|)|v| \in L^1(\Omega).
	\end{equation*}
	Compute $\frac{\partial g}{\partial \epsilon}(\epsilon, \delta)$ and $\frac{\partial g}{\partial \delta}(\epsilon, \delta)$, we get
	$$
	\begin{aligned}
		& \frac{\partial g}{\partial \epsilon}(\epsilon, \delta)=-\int_{\Omega}\omega_{1} \psi((1-\epsilon) \left|u\right|+\delta \left|v\right|)\left| u\right| , \\
		& \frac{\partial g}{\partial \delta}(\epsilon, \delta)=\int_{\Omega} \omega_{1}\psi((1-\epsilon) \left|u\right|+\delta \left|v\right|) \left| v\right|  .
	\end{aligned}
	$$
	Again, by the dominated convergence theorem, $g \in C^1(D)$. Finaly, since $g(0,0)=0$ and $\frac{\partial g}{\partial \delta}(0,0)=\int_{\Omega}\omega_{1}\psi(\left|u\right| )v \neq 0$, then by the implicit function theorem there exists a continuously differentiable function $\delta$ defined in some interval $\left(-\epsilon_0,\epsilon_0\right)$ satisfying
	\begin{equation}
		g(\epsilon, \delta(\epsilon))=0 \quad \text { for all }-\epsilon_0<\epsilon<\epsilon_0
	\end{equation}
	and
	\begin{equation*}
		\delta^{\prime}(0)=-\frac{\frac{\partial g}{\partial \epsilon}(0,0)}{\frac{\partial g}{\partial \delta}(0,0)}=\frac{\int_{\Omega} \omega_{1}\psi(\left| u\right| ) \left| u\right| }{\int_{\Omega} \omega_{1}\psi(\left| u\right| ) \left| v\right| }.
	\end{equation*}
\end{proof}
The proof is complete.
\begin{Lemma}\label{algebric4}
	If $u_{\alpha}\in W_{0,\omega,\omega_{1}}^{1}L^{\Phi,\Psi}\left(\Omega\right)$ solves the minimisation Problem (\ref{minid}), then $\varphi(\nabla u_{\alpha})\in L_{\omega}^{\tilde{\Phi}}(\Omega)$.
\end{Lemma}
\begin{proof}[\textbf{Proof}]
	The proof run analogously to the proof of Proposition 4.3 in \cite{mustonen1999eigenvalue} replacing the function $\Psi$ in \cite{mustonen1999eigenvalue}[Proposition 4.3] by a function
	\begin{equation*}
		Z(\epsilon)=\int_{\Omega}\omega\Phi\left(\left(1-\epsilon\right)\left|\nabla u\right|+\delta(\epsilon)\left|\nabla u\right|\right)
	\end{equation*}
	where $\delta$ is the continuously differentiate function apeared in Lemma \ref{delta}.
\end{proof}
Now we are in position to prove Theorem \ref{T2}.
\begin{proof}[\textbf{Proof of Theorem \ref{T2}}]
	Define functionals $DJ$ and $DI$ by
	\begin{equation*}
		\langle DJ,v\rangle:=\int_{\Omega}\omega(x)\frac{\varphi\left(\nabla u_{\alpha}\right)}{\left|\nabla u_{\alpha}\right| }\nabla u_{\alpha}\nabla  vdx
	\end{equation*}
	\begin{equation*}
		\langle DI,v\rangle:=\int_{\Omega}\omega_{1}(x)\psi(\left|u_{\alpha}\right|)\frac{u_{\alpha}}{\left|u_{\alpha}\right|}vdx
	\end{equation*}
	We already know from above that $DI,  DJ \in \left( W^{1}_{0,\omega,\omega_{1}}E^{\Phi,\Psi}(\Omega)\right)^{\prime}$. According to Proposition \ref{KER},we shall prove that $\ker DJ\subset\ker DI$ to prove the existance of $\lambda=\lambda_{\alpha}$ such that for all $v\in W^{1}_{0,\omega,\omega_{1}}E^{\Phi,\Psi}(\Omega)$
	\begin{equation}\label{solution}
		\langle DI,v\rangle=\lambda \langle DJ,v\rangle.
	\end{equation}
	Since $W^{1}_{0,\omega,\omega_{1}}E^{\Phi,\Psi}(\Omega)$ is $\sigma(W^{1}_{0,\omega,\omega_{1}}L^{\Phi,\Psi}(\Omega), \left( W^{1}_{0,\omega,\omega_{1}}E^{\Phi,\Psi}(\Omega)\right)^{\prime})$ dense in $W^{1}_{0,\omega,\omega_{1}}L^{\Phi,\Psi}(\Omega)$, (\ref{solution}) holds also for all $v \in W^{1}_{0,\omega,\omega_{1}}L^{\Phi,\Psi}(\Omega)$. 
	To this end it is sufficient to show that $V_{J}\subset V_{I} $, where $V_{J}$ and $V_{I}$ are defined by
	\begin{equation*}
		V_{I}:=\left\lbrace v\in W_{0,\omega,\omega_{1}}^{1}E^{\Phi,\Psi}\left(\Omega\right);\ \int_{\Omega}\omega(x)\frac{\varphi\left(\nabla u_{\alpha}\right)}{\left|\nabla u_{\alpha}\right| }\nabla u_{\alpha}\nabla  vdx>0 \right\rbrace
	\end{equation*}
	\begin{equation*}
		V_{J}:=\left\lbrace v\in W_{0,\omega,\omega_{1}}^{1}E^{\Phi,\Psi}\left(\Omega\right);\ \int_{\Omega}\omega_{1}(x)\psi(\left|u_{\alpha}\right|)\frac{u_{\alpha}}{\left|u_{\alpha}\right|}vdx>0 \right\rbrace.
	\end{equation*}
	By Lemma \ref{delta}, there exists $\epsilon_{0}>0$ and $\delta\in C^{1}\left(-\epsilon_{0},\epsilon_{0}\right)$ such that for all $\epsilon\in\left( -\epsilon_0,\epsilon_0\right)$, we have
	\begin{equation*}
		\int_{\Omega}\omega_{1}\Psi\left(\left(1-\epsilon\right)\left|u_{\alpha}\right|+\delta v\right)dx=\int_{\Omega}\omega_{1}\Psi\left(\left|u_{\alpha}\right|\right)=\alpha.
	\end{equation*}
	and
	\begin{equation*}
		\delta^{\prime}(0)=\frac{\int_{\Omega} \omega_{1}\psi(\left| u_{\alpha}\right| ) \left| u_{\alpha}\right| }{\int_{\Omega} \omega_{1}\psi(\left| u_{\alpha}\right| ) \left| v\right| },
	\end{equation*}
	which is positive quantities by the definition of $V_{I}$. For this reason we may choose $0<\epsilon_{1}<\epsilon_{0}$, such that
	$$
	\frac{1}{2} \delta^{\prime}(0)<\delta^{\prime}(\epsilon)<2 \delta^{\prime}(0) \quad \text { for all }-\epsilon_1<\epsilon<\epsilon_1 \text {. }
	$$
	Hence, by integration, we get
	$$
	\frac{1}{2} \delta^{\prime}(0)<\frac{\delta(\epsilon)}{\epsilon}<2 \delta^{\prime}(0) \quad \text { for all } 0<\epsilon<\epsilon_1 .
	$$
	Denote $w_\epsilon=(1-\epsilon) u_\alpha+\delta(\epsilon) v$. By the definition of $u_\alpha$ we have
	\begin{equation}\label{4.4}
		\int_{\Omega} \frac{\omega \Phi\left(\left|\nabla w_\epsilon\right|\right)-\omega\Phi\left(\left|\nabla u_\alpha\right|\right)}{\delta(\epsilon)} \geq 0 \quad \text { for all } 0<\epsilon<\epsilon_1 \text {. }
	\end{equation}
	Since $\nabla w_\epsilon \rightarrow \nabla u_\alpha$ a.e. in $\Omega$, when $\epsilon \rightarrow 0^{+}$, direct calculation gives
	$$
	\begin{aligned}
		\frac{\Phi\left(\left|\nabla w_\epsilon\right|\right)-\Phi\left(\left|\nabla u_\alpha\right|\right)}{\delta(\epsilon)} & =\frac{ \Phi\left(\left|\nabla w_\epsilon\right|\right)- \Phi\left(\left|\nabla u_\alpha\right|\right)}{\left|\nabla w_c\right|-\left|\nabla u_\alpha\right|} \frac{\left|\nabla w_\epsilon\right|^2-\left|\nabla u_\alpha\right|^2}{\left(\left|\nabla w_c\right|+\left|\nabla u_\alpha\right|\right) \delta(\epsilon)} \\
		& \rightarrow \varphi\left(\left|\nabla u_\alpha\right|\right)\left(-\frac{1}{\delta^{\prime}(0)}\left|\nabla u_\alpha\right|+\frac{\nabla u_\alpha \cdot \nabla v}{\left|\nabla u_\alpha\right|}\right)
	\end{aligned}
	$$
	a.e. in $\Omega$, when $\epsilon \rightarrow 0^{+}$. Moreover, by monotonicity and the triangle inequality, we have
	$$
	\begin{aligned}
		\left|\frac{\Phi\left(\left|\nabla w_\epsilon\right|\right)-\Phi\left(\left|\nabla u_\alpha\right|\right)}{\delta(\epsilon)}\right| & \leq\left(\varphi\left(\left|\nabla w_\epsilon\right|\right)+\varphi\left(\left|\nabla u_\alpha\right|\right)\right) \frac{\left|\nabla w_\epsilon-\nabla u_\alpha\right|}{\delta(\epsilon)} \\
		& \leq\left(2 \varphi\left(\left|\nabla u_\alpha\right|\right)+\varphi\left(\frac{\delta(\epsilon)}{\epsilon}|\nabla v|\right)\right)\left(\frac{\epsilon}{\delta(\epsilon)}\left|\nabla u_\alpha\right|+|\nabla v|\right) \\
		& \leq\left(2 \varphi\left(\left|\nabla u_\alpha\right|\right)+\varphi\left(2 \delta^{\prime}(0)|\nabla v|\right)\right)\left(\frac{2}{\delta^{\prime}(0)}\left|\nabla u_\alpha\right|+|\nabla v|\right) \\
		& \in L^1(\Omega) .
	\end{aligned}
	$$
	By the dominated convergence theorem,
	$$
	\frac{\Phi\left(\left|\nabla w_\epsilon\right|\right)-\Phi\left(\left|\nabla u_\alpha\right|\right)}{\delta(\epsilon)} \rightarrow-\frac{1}{\delta^{\prime}(0)} \varphi\left(\left|\nabla u_\alpha\right|\right)\left|\nabla u_\alpha\right|+\frac{\varphi\left(\left|\nabla u_\alpha\right|\right)}{\left|\nabla u_\alpha\right|} \nabla u_\alpha \cdot \nabla v
	$$
	in $L^1(\Omega)$ as $\epsilon \rightarrow 0^{+}$. In view of (\ref{4.4}), we conclude
	$$
	\int_{\Omega} \omega\frac{\varphi\left(\left|\nabla u_\alpha\right|\right)}{\left|\nabla u_\alpha\right|} \nabla u_\alpha \cdot \nabla v \geq \frac{1}{\delta^{\prime}(0)} \int_{\Omega} \omega\varphi\left(\left|\nabla u_\alpha\right|\right)\left|\nabla u_\alpha\right|>0
	$$
	implying $v\in V_{I}$. Since $v$ was arbitrary, $V_{J} \subset V_I$ follows. 
\end{proof}     
\begin{proof}[\textbf{Proof of Theorem \ref{T3}}]
	Let $\left\{\lambda_k\right\}_{k \in \mathbb{N}}$ be sequence of eigenvalues of (\ref{mainproblem0}) such that $\lambda_k \rightarrow \lambda$ and let $\left\{u_k\right\}_{k \in \mathbb{N}} \subset W_{0,\omega,\omega_{1}}^{1}L^{\Phi,\Psi}\left(\Omega\right)$ be the corresponding sequence of associated eigenfunctions, i.e.,
	$$
	\int_{\Omega} \omega\varphi\left(\left|\nabla u_k\right|\right) \frac{\nabla u_k}{\left|\nabla u_k\right|} \nabla v d \mu=\lambda_k \int_{\Omega} \omega_{1}\psi\left(\left|u_k\right|\right) \frac{u_k}{\left|u_k\right|} v \quad \forall v \in W_{0,\omega,\omega_{1}}^{1}L^{\Phi,\Psi}\left(\Omega\right) .
	$$
	Arguing as in the proof of Proposition \ref{4.2}, up to a subsequence, there exists $u \in W_{0,\omega,\omega_{1}}^{1}L^{\Phi,\Psi}\left(\Omega\right)$ such that
	$$
	\begin{array}{ll}
		u_k \longrightarrow u & \text { strongly in } L_{\omega_{1}}^\Psi(\Omega), \\
		u_k \longrightarrow u & \text { a.e. in } \mathbb{R}^n .
	\end{array}
	$$
	From the continuity of $t \mapsto \varphi(t) \frac{t}{|t|}$ and $t \mapsto \psi(t) \frac{t}{|t|}$  we deduce that
	$$
	\varphi\left(\left|\nabla u_k\right|\right) \frac{\nabla u_k}{\left|\nabla u_k\right|} \longrightarrow \varphi\left(\left|\nabla u\right|\right) \frac{\nabla u}{\left|\nabla u\right|} \quad \text { a.e. in } \Omega
	$$
	and
	$$
	\psi\left(\left|u_k\right|\right) \frac{u_k}{\left|u_k\right|} \longrightarrow \psi\left(\left| u\right|\right) \frac{ u}{\left| u\right|} \quad \text { a.e. in } \Omega
	$$
	Hence, taking limit as $k \rightarrow \infty$ in (5.1) we obtain that
	$$
	\int_{\Omega} \omega\varphi\left(\left|\nabla u\right|\right) \frac{\nabla u}{\left|\nabla u\right|} \nabla v d \mu=\lambda \int_{\Omega}\omega_{1} \psi(|u|) \frac{u}{|u|} v \quad \text { for all } v \in W_{0,\omega,\omega_{1}}^{1}L^{\Phi,\Psi}\left(\Omega\right).
	$$
	The proof is complete.
\end{proof}
\begin{Remark}
	The Problem (\ref{mainproblem0}) can be studied in $W^{1}_{0,\omega}L^{\Phi}(\Omega)$ for the same reason of Remark \ref{remark}	
\end{Remark}

\appendix
\section{An abstract existence result}

\begin{Proposition} \cite{gasinski2005nonsmooth}\label{compact}
	If $X$ is a reflexive Banach space, $Y$ is a Banach space, $Z\subset X$ is nonempty, closed and convex and $J:Z\rightarrow Y$ is completely continuous, then $J$ is compact.
\end{Proposition}

\begin{Theorem}\cite[Theorem 3.6]{bonano}\label{bonano}
	Let X be a reflexive real Banach space and $I:X\rightarrow\mathbb{R}$ a coercive, continuously Gateaux differentiable and sequentially weakly lower semicontinuous functional whose Gateaux derivative admits a continuous inverse on $X$.
	Let $J:X\rightarrow \mathbb{R}$ be a continuously Gateaux differentiable functional whose Gateaux derivative is compact such that
	$$(a_0)\quad\inf_{x\in X}I(x)=I(0)=J(0)=0.$$
	Assume that there exist $r > 0$ and $\overline{x}\in X$, with $r < I(\overline{x})$, such that:
	$$(a_1)\quad\frac{\displaystyle\sup_{I(x)\leq r}J(x)}{r}<\frac{J(\overline{x})}{I(\overline{x})};$$
	$$(a_2)\quad
	\mbox{for each}\;
	\lambda\in\Lambda_r:=\Big(\frac{I(\overline{x})}{J(\overline{x})}, \frac{r}{\displaystyle\sup_{I(x)\leq r}J(x)}\Big),
	\mbox{the functional}\;
	I-\lambda J
	\;
	\hbox{is coercive}.$$
	Then for each $\lambda\in\Lambda_r$, the functional $I-\lambda J$ has at least three distinct critical points in $X$.
\end{Theorem}

\section{Ljusternik-Schnirelman}
We shall specifically refer to the abstract theorem of the so-called Ljusternik-Schnirelman theory  from \cite[Theorem 9.27]{motreanu2014topological}. See also \cite{browder1965lusternik, motreanu2014topological, zeidler2013nonlinear}.
\begin{Theorem}\cite[Theorem 9.27]{motreanu2014topological}\label{B1}
	Given $\alpha>0$, assume that $A$, $B$ are two functionals defined in a reflexive Banach space $X$, such that
	\begin{itemize}
		\item[($h_1$)] $A$,$B$  are  $C^1(X,\R)$ even functionals  with $A(0)=B(0)=0$ and the level set
		$$
		M_\alpha := \{ u\in X \colon B(u)=\alpha\}
		$$
		is bounded.
		
		\item [($h_2$)] $A'$ is strongly continuous, i.e.,
		$$
		u_n\rightharpoonup u \text{ in } X \implies A'(u_n) \to I'(u).
		$$
		Moreover, for any $u$ in the closure of the convex hull of $M_\alpha$,
		$$
		\langle A'(u),u\rangle=0 \iff A(u)=0 \iff u=0.
		$$
		
		\item [($h_3$)] $B'$ is continuous, bounded and, as $n\to\infty$, it holds that
		$$
		u_n\rightharpoonup u, \quad B'(u_j)\rightharpoonup v, \quad \langle B'(u_n),u_n\rangle\to\langle v,u\rangle \implies u_n\to u \text{ in } X.
		$$
		
		\item [($h_4$)] For every $u \in X \setminus\{0\}$ it holds that
		$$
		\langle B'(u),u\rangle>0,\qquad \displaystyle\lim_{t\to+\infty}B(tu)=+\infty,\qquad \displaystyle\inf_{u\in M_\alpha}\langle B'(u),u\rangle>0.
		$$
	\end{itemize}
	
	Define max-min values
	\begin{equation*}\label{ak}
		c_{k,\alpha}=\begin{cases}
			\sup_{K\in \mathcal{C}_k}\inf_{u\in K}A(u), & \mathcal{C}_k\neq\emptyset , \\
			0, &  \mathcal{C}_k= \emptyset,
		\end{cases}
	\end{equation*}
	where, for any $k\in\mathbb{N}$,
	$$
	\mathcal{C}_k:=\{ K\subset M_\alpha \text{ compact, symmetric with } A(u)>0 \text{ on } K  \text{ and } \gamma(K)\geq k\},
	$$
	and the Krasnoselskii genus of $K$ is defined as
	$$
	\gamma(K):=\inf\{p\in\mathbb{N}\colon \exists h\colon K\to \R^p\setminus\{0\}\ \text{such that}\ h \text{ is continuous and odd} \},
	$$
	see \cite{yu1964topological} for details.
	
	Thus,  $\{c_{k,\alpha}\}_{k\geq1}$ forms a nonincreasing sequence
	$$
	+\infty> c_{1,\alpha}\geq c_{2,\alpha}\geq \ldots\geq c_{k,\alpha}\geq\ldots\geq0.
	$$
	
	Under these considerations, there exists a sequence $\{(\mu_{k,\alpha},u_{k,\alpha})\}_{k\geq1}$ such that
	$$
	\langle A'(u_{k,\alpha}),v\rangle=\mu_{k,\alpha} \langle B'(u_{k,\alpha}),v\rangle \qquad \forall v\in X
	$$
	such that $u_{k,\alpha} \in M_\alpha$, $A(u_{k,\alpha})=c_{k,\alpha}$,  $\mu_{k,\alpha}\neq 0$, $\mu_{k,\alpha}\to 0$, and $u_{k,\alpha}\rightharpoonup 0$ in $X$.
\end{Theorem}
\section{ An abstract Eigenvalues Problem without $\Delta_{2}$}
\begin{Proposition}\label{KER}\cite{zeidler2013nonlinear}
	Assume that the following two conditions hold:\\
	(i) $X$ and $Y$ are $B$-spaces over $\mathbb{K}$, where $\mathbb{K}=\mathbb{R}$ or $\mathbb{C}$.\\
	(ii) $A: X \rightarrow Y$ and $B: X \rightarrow \mathbb{K}$ are continuous linear operators and $R(A)$ is closed.\\
	Then if
	$$
	B h=0 \text { for all } h \in X \text { such that } A h=0
	$$
	holds, there exists $a \Lambda \in Y^*$ such that
	$$
	\lambda_0 B k+\Lambda(A k)=0 \quad \text { for all } k \in X,
	$$
	with $\lambda_0=1$. For $R(A)=Y, \Lambda$ is unique.
\end{Proposition}
\section*{Acknowledgement(s)}
We thank the referees for their times and comments
\section*{Disclosure statement}
No potential conflict of interest was reported by the author(s).

%

%


\begin{thebibliography}{99}
\bibitem{zeidler2013nonlinear} Zeidler, Eberhard Nonlinear functional analysis and its applications: III: variational methods and optimization,
\end{thebibliography}
\end{document}